\documentclass[final,onefignum,onetabnum]{siamonline190516}
\usepackage[scale=0.8]{geometry} 

\usepackage{adjustbox}
\usepackage{lipsum}
\usepackage{amsfonts}
\usepackage{graphicx}
\usepackage{epstopdf}
\usepackage{algorithmic}
\ifpdf
  \DeclareGraphicsExtensions{.eps,.pdf,.png,.jpg}
\else
  \DeclareGraphicsExtensions{.eps}
\fi

\usepackage{datetime}
\newdateformat{monthyeardate}{%
	\monthname[\THEMONTH] \THEDAY, \THEYEAR}

\usepackage{academicons}
\usepackage{xcolor}
\renewcommand{\orcid}[1]{\href{https://orcid.org/#1}{\textcolor[HTML]{A6CE39}{orcid.org/#1}}}

\usepackage{amsmath}
\allowdisplaybreaks
\usepackage{amssymb}
\usepackage{commath}
\usepackage{mathtools}
\usepackage{bbm}
\usepackage{bm}
\usepackage{physics}
\usepackage{theoremref}
\usepackage{array}
\usepackage[normalem]{ulem}
\usepackage{color}
\usepackage{graphicx}
\usepackage[small]{caption}
\usepackage{subcaption}
\usepackage{booktabs}

\usepackage{enumitem}
\setlist[enumerate]{leftmargin=.5in}
\setlist[itemize]{leftmargin=.5in}

\newsiamremark{remark}{Remark}
\newsiamremark{hypothesis}{Hypothesis}
\crefname{hypothesis}{Hypothesis}{Hypotheses}
\newsiamthm{claim}{Claim}
\newsiamremark{goal}{Goal}

\headers{Efficient sparsity-promoting MAP estimation}{J.\ Lindbloom, J.\ Glaubitz, and A.\ Gelb}

\title{Efficient sparsity-promoting MAP estimation for Bayesian linear inverse problems
\thanks{\monthyeardate\today 
\corresponding{Jonathan Lindbloom} 
}}

\author{ 
Jonathan Lindbloom\thanks{
Department of Mathematics, Dartmouth College, USA (\email{Jonathan.T.Lindbloom.GR@Dartmouth.edu}, \orcid{0000-0002-1789-2629} and \email{Anne.E.Gelb@Dartmouth.edu}, \orcid{0000-0002-9219-4572})}
\and 
{Jan Glaubitz\thanks{
Department of Aeronautics and Astronautics, Massachusetts Institute of Technology, USA 
(\email{glaubitz@mit.edu}, \orcid{0000-0002-3434-5563})
}} \thanks{
Department of Mathematics, Link\"oping University, Sweden 
}
\and 
Anne Gelb\footnotemark[2]
}

\usepackage{amsopn}
\usepackage{cancel}

\input{jlmacros}

\ifpdf
\hypersetup{
  pdftitle={Generalizing IAS},
  pdfauthor={Lindbloom, Gelb, and Glaubitz}
}
\fi

\crefformat{model}{Model ~(#2#1#3)}

\begin{document}

\maketitle

\begin{abstract}
Bayesian hierarchical models can provide efficient algorithms for finding sparse solutions to ill-posed linear inverse problems. 
The models typically comprise a conditionally Gaussian prior model for the unknown augmented by a generalized gamma hyper-prior model for the variance hyper-parameters. 
This investigation generalizes such models and their efficient maximum a posterior (MAP) estimation using the iterative alternating sequential (IAS) algorithm in two ways:
(1) General sparsifying transforms: 
Diverging from conventional methods, our approach permits  use of sparsifying transformations with nontrivial kernels;
(2) Unknown noise variances: 
The noise variance is treated as a random variable to be estimated during the inference procedure. This is important in applications where the noise estimate cannot be accurately estimated \emph{a priori}. 
Remarkably, these augmentations neither significantly burden the computational expense of the algorithm nor compromise its efficacy. We include convexity and convergence analysis and demonstrate our method's efficacy in several numerical experiments. 

\end{abstract}

\begin{keywords}
    Image reconstruction, 
    Bayesian inverse problems, 
    sparsity-promoting hierarchical Bayesian learning, 
    conditionally Gaussian priors, 
    (generalized) gamma hyper-priors, 
    convexity 
\end{keywords}

\begin{AMS} 
    62F15, 
    65F22, 
    65K10, 
    68Q25, 
    68U10 
\end{AMS}

\begin{Code}
    \url{https://github.com/jlindbloom/GeneralizedSparsitySolvers} 
\end{Code}

\begin{DOI}
\url{https://doi.org/10.1088/1361-6420/ada17f}
\end{DOI}

\section{Introduction}
\label{sec:intro}

Recovering a sparse parameter vector from indirect, incomplete, and noisy observations is a common yet challenging problem in a variety of applications. 
The task is often modeled as a linear inverse problem
\begin{equation}\label{eq:intro_IP}
	\mathbf{y} = F \mathbf{x} + \mathbf{e}, 
\end{equation}
where $\mathbf{y} \in \R^M$ is a vector of observations,  $\mathbf{x} \in \R^N$ symbolizes the unknown parameter vector, $F \in \R^{M \times N}$ is the known linear forward operator, and $\mathbf{e} \in \R^M$ corresponds to the noise component. 
Comprehensive discussions on inverse problems may be found in \cite{groetsch1993inverse,vogel2002computational,hansen2010discrete} and related references. 
In particular, \cref{eq:intro_IP} may be associated with signal or image reconstruction \cite{hansen2006deblurring,stark2013image}. 
If $F$ is ill-conditioned or if the data are significantly distorted by noise, then  \cref{eq:intro_IP} becomes ill-posed and pathologically hard to solve.

Prior knowledge about the otherwise unknown parameter vector $\mathbf{x}$ is often leveraged to overcome the associated challenges. 
In this regard, using a Bayesian approach \cite{kaipio2006statistical,stuart2010inverse,calvetti2023bayesian}, which models the parameter and observation vectors as random variables, is known to be highly successful. 
In a nutshell, the sought-after posterior distribution for the parameters of interest is characterized using Bayes' theorem, which connects the posterior density to the prior and likelihood densities.
The prior encodes information available on the parameters of interest before any data are observed. At the same time, the likelihood density incorporates the data model \cref{eq:intro_IP} and a stochastic description of the measurements. 
This investigation assumes that some linear transformation of the parameter vector, say $R \mathbf{x}$, is sparse. 
That is, most components of $R \mathbf{x}$ are approximately zero. 
For instance, $\mathbf{x}$ might correspond to the nodal values of a piecewise constant signal, in which case $R$ would be a discrete gradient operator. A particularly potent class of sparsity-promoting priors are those that can be decomposed into a conditional Gaussian prior and a generalized gamma hyper-prior. 
These have been proven successful in various applications as they are computationally convenient and often produce highly efficient inference algorithms.  See  \cite{calvetti2019hierachical,calvetti2020sparse,calvetti2020sparsity,calvetti2023bayesian} and \cite{Tipping2001SparseBL,wipf2004sparse,chantas2006bayesian,babacan2010sparse,glaubitz2022generalized,xiao2023sequential}, where the variance and precision of the conditionally Gaussian prior was equipped with a (generalized) gamma hyper-prior, respectively, and references therein.

Although recent advances in \cite{calvetti2024computationally} initiated the development of sampling strategies for sparsity-promoting hierarchical models, many algorithms still primarily focus on obtaining MAP estimates through the use of a block coordinate descent method, often referred to as the IAS algorithm. 
See \cite{calvetti2007gaussian,calvetti2015hierarchical,calvetti2019hierachical,calvetti2020sparse} and references therein. While here the MAP approach may be compared with other traditional total variation (TV) methods, the conditional Gaussian structure of the hierarchical model provides an efficient framework to later perform sample-based uncertainty quantification. The basic idea behind the IAS algorithm is to alternatingly update the parameters of primary interest $\mathbf{x}$ and the hyper-parameters $\boldsymbol{\theta}$, which encode the sparsity profile of $\mathbf{x}$, while keeping the other set of parameters fixed. 
Among the two updates, the one concerning $\mathbf{x}$ is computationally more demanding and involves the solution of a least squares problem. This cost can be mitigated by preconditioning strategies such as priorconditioning. This procedure typically presumes that the sparsifying transform $R$ is either invertible or possesses a trivial kernel, i.e., $\ker(R) = \{ \mathbf{0}_N \}$. 
In cases where $R$ has a non-trivial kernel, the conditional prior becomes improper and computational challenges arise. Building upon the convexity analysis provided in \cite{calvetti2020sparse}, hybrid solvers that switch between globally and locally convex models were developed in \cite{calvetti2020sparsity,si2024path}. 
Notably, the existing convexity analyses, although seamlessly applicable to the case of an invertible sparsifying transformation $R$, do not readily extend to noninvertible transformations.

\subsection*{Sparsifying transforms with non-trivial kernel}

This investigation expands the utility of the IAS algorithm by embracing more versatile sparsifying transformations with nontrivial kernels. 
Specifically, we provide convexity and convergence analyses for this generalization of the IAS algorithm. Furthermore, we provide details for its efficient computational implementation. This enhancement allows us to employ various discrete gradient operators not previously permitted by IAS methods, such as anisotropic and isotropic operators with Neumann boundary conditions, across multiple dimensions. Our approach also eliminates the need to impose artificial boundary conditions that may be neither available nor suitable. Finally, such generalization enables the incorporation of other sparsifying transforms, such as wavelet \cite{rioul1991wavelets} or polynomial annihilation \cite{archibald2005polynomial,archibald2016image} operators. 

\subsection*{Variable noise variances}

Further enriching our framework, we extend the IAS algorithm to treat the noise variance $\sigma^2$---an essential component of the data model \cref{eq:intro_IP}---as a random variable. 
This adaptation is critical in numerous applications where the noise variance is either imprecise or subject to fluctuations during the device's operational lifespan. 
Rather than merely representing the uncertainty about the noise variance, the corresponding random variable also encapsulates model discrepancies, a recurrent issue in almost all realistic scenarios. 
Notably, evidence from similar models indicates that employing a random variable can bolster results, even when the exact noise variance is known \cite{zhang2011clarify}. 
We derive an additional yet computationally efficient update step for the correspondingly adapted IAS algorithm by modeling the noise variance as generalized gamma-distributed. 
We also provide convexity and convergence analyses for the resulting IAS method.

\subsection*{Outline} 

We review the state of the art of the IAS framework in \Cref{sec:background}. 
In \Cref{sec:main_results} we generalize the IAS algorithm to accommodate sparsifying transforms with nontrivial kernels and unknown noise variances. Numerical experiments, including a computed tomography (CT) inverse problem, are presented in \Cref{sec:numerics}. Finally, \Cref{sec:summary} provides some concluding remarks.

\section{Preliminaries}
\label{sec:background} 

We first review the class of sparsity-promoting hierarchical models developed in \cite{calvetti2007gaussian,calvetti2009conditionally,calvetti2015hierarchical} and the IAS algorithm for their MAP estimation. 
Also see the more recent developments in \cite{calvetti2019hierachical,calvetti2020sparse,calvetti2020sparsity,glaubitz2024leveraging} and references therein. While this section primarily serves to review existing methods along with their properties,  \cref{thm:IAS_G_convergence_convex_case} and \cref{thm:IAS_G_limit_points} provide some new convergence results.

\subsection{Notation and nomenclature} 

We use $x_i$ to denote the $i$th component of $\mathbf{x} \in \mathbb{R}^N$, and $D_{\mathbf{x}}$ or $\operatorname{diag}(\mathbf{x})$ to denote the $N \times N$ diagonal matrix with $\mathbf{x}$ on its diagonal. When convenient, we denote a vector in terms of its entries as $\mathbf{x} = \operatorname{vec}(x_i)$. We write $\mathbb{R}^N_{+}$ and $\mathbb{R}_{++}^N$ for the nonnegative and positive orthants (containing the vectors with only nonnegative and positive entries), respectively.
For a given set $S$, we denote by $\operatorname{Int}(S)$ its interior and  by $\operatorname{Bd}(S)$ its boundary, along with its corresponding characteristic function $\delta_S$ with $\delta_S(\mathbf{x}) = 0$ if $\mathbf{x} \in S$ and $\delta_S(\mathbf{x}) = +\infty$ otherwise, as well as the indicator function $\mathbbm{1}_{S}$ with $\mathbbm{1}_{S}(\mathbf{x}) = 1$ if $\mathbf{x} \in S$ and $\mathbbm{1}_{S}(\mathbf{x}) = 0$ otherwise. We call a point $\mathbf{x}^* \in \mathbb{R}^N$ a stationary point of an extended real-valued function $\mathcal{F}:\mathbb{R}^N \to \mathbb{R} \cup \{ + \infty \}$ if $\mathbf{0}_N \in \partial \mathcal{F}(\mathbf{x}^*)$, where $\partial \mathcal{F}$ is the \emph{limiting subdifferential} or \emph{subdifferential} (see \cite{Mordukhovich2006}) of $\mathcal{F}$ at $\mathbf{x}^*$. If $\mathcal{F}$ is differentiable at a point $\mathbf{x}$, then we simply have $\partial \mathcal{F}(\mathbf{x}) = \{ \nabla \mathcal{F}(\mathbf{x}) \}$. We employ this definition of the subdifferential since, for our purposes, $\mathcal{F}$ may be nonconvex.

\subsection{The hierarchical Bayesian model}
\label{sub:background_model} 

Consider the linear data model \cref{eq:intro_IP} with independent and identically distributed (i.i.d.) zero-mean normal noise $\mathbf{e}$. Assume further that the parameter vector $\mathbf{x}$ is sparse. 
In a hierarchical Bayesian setting, this motivates the generative model  
\begin{equation}\label{eq:IAS_model} 
\begin{aligned}
	\mathbf{y} | \mathbf{x} 
		& \sim \mathcal{N}( F \mathbf{x}, \sigma^2 I ), \\ 
	\mathbf{x} | \boldsymbol{\theta} 
		& \sim \mathcal{N}(\mathbf{0}_N,D_{\boldsymbol{\theta}}), \\ 
	\theta_i 
		&\overset{\text{ind}}{\sim} \textrm{GG}(r,\beta,\vartheta_i), \quad i=1,\dots,N,   
\end{aligned}
\end{equation} 
for a noise variance parameter $\sigma^2$ and hyper-parameters $\boldsymbol{\theta} = [\theta_1,\dots,\theta_N]^T \in \mathbb{R}_{+}^N$. Here $ \textrm{GG}(r, \beta, \vartheta)$ denotes the \emph{generalized Gamma distribution} with density 
\begin{align}
    \pi(\theta) = \frac{|r|}{\Gamma(\beta)} \left( \frac{\theta}{\vartheta} \right)^{r \beta - 1} \exp\left( - \left( \frac{\theta}{\vartheta} \right) \right)^r \mathbbm{1}_{\mathbb{R}_{+}}(\theta),
\end{align}
defined for parameters $r \in \mathbb{R} \setminus \{ 0 \}$, $\beta > 0$, and $\vartheta > 0$. 
Following \cite{calvetti2007gaussian,glaubitz2022generalized}, the combination of a conditional Gaussian prior and a generalized gamma hyper-prior promoting sparsity can be understood as follows: 
The generalized gamma hyper-prior is centered at zero but allows for occasional outliers. 
A typical realization $\boldsymbol{\theta}$ from this distribution will have mostly small components, with some significantly larger than zero. 
If $\theta_i \approx 0$, then the conditionally Gaussian prior for the $i$th component of $\mathbf{x}$, which has distribution $\mathcal{N}(0, \theta_i)$,  favors $x_i \approx 0$ since such an $x_i$ has a higher probability. 
However, if $\theta_i$ is one of the few components significantly larger than zero, then $| x_i | \gg 0$ becomes more likely. 

By Bayes' theorem, the joint posterior density function is given by $\pi( \mathbf{x}, \boldsymbol{\theta} | \mathbf{y} ) 
		\propto \pi( \mathbf{y} | \mathbf{x} ) \, 
			\pi( \mathbf{x} | \boldsymbol{\theta} ) \,  
			\pi(\boldsymbol{\theta})$, i.e., the product of the likelihood, prior, and hyper-prior densities. 
According to \cref{eq:IAS_model}, the posterior density can be expressed as
\begin{equation}\label{eq:IAS_posterior} 
	\pi( \mathbf{x}, \boldsymbol{\theta} | \mathbf{y} ) 
		\propto \exp\left( - \frac{1}{2 \sigma^2} \| F \mathbf{x} - \mathbf{y} \|_2^2 - \frac{1}{2} \| D_{\boldsymbol{\theta}}^{-1/2} \mathbf{x} \|_2^2 - \sum_{i=1}^N \left( \frac{\theta_i}{\vartheta_i} \right)^r \right) \det( D_{\boldsymbol{\theta}} )^{r \beta - 3/2} \mathbbm{1}_{ \mathbb{R}_{+}^N}(\boldsymbol{\theta}).
\end{equation}

\subsection{MAP estimation and the IAS algorithm}
\label{sub:background_MAP}

We now address Bayesian inference for \cref{eq:IAS_model}. 
To this end, a common strategy is to solve for the \emph{MAP estimate} $(\mathbf{x}^{\MAP},\boldsymbol{\theta}^{\MAP})$ for given measurements $\mathbf{y}$, which is the maximizer of the joint posterior density \cref{eq:IAS_posterior}. 
Equivalently, the MAP estimate is the minimizer of the negative logarithm of the posterior, i.e.,  
\begin{equation}\label{eq:IAS_MAP_estimate}
	(\mathbf{x}^{\MAP},\boldsymbol{\theta}^{\MAP}) 
		= \argmin_{ \mathbf{x}, \boldsymbol{\theta} } \left\{ \mathcal{G}( \mathbf{x}, \boldsymbol{\theta} ) \right\}, 
\end{equation} 
with objective function (also called the \emph{Gibbs energy functional} or \emph{potential}) given by $\mathcal{G}( \mathbf{x}, \boldsymbol{\theta} ) = - \log \pi( \mathbf{x}, \boldsymbol{\theta} | \mathbf{y} )$. Substituting \cref{eq:IAS_posterior} into $\mathcal{G}$ yields  
\begin{equation}\label{eq:IAS_G}
	\mathcal{G}( \mathbf{x}, \boldsymbol{\theta} ) 
		= \frac{1}{2 \sigma^2} \| F \mathbf{x} - \mathbf{y} \|_2^2 + \frac{1}{2} \| D_{\boldsymbol{\theta}}^{-1/2} \mathbf{x}  \|_2^2 + \sum_{i=1}^N \left( \frac{\theta_i}{\vartheta_i} \right)^r - \sum_{i=1}^N ( r \beta - 3/2 ) \log( \theta_i ) + \delta_{\mathbb{R}_{+}^N}( \boldsymbol{\theta}) 
\end{equation}
up to constants that depend neither on $\mathbf{x}$ nor $\boldsymbol{\theta}$. 
Here, we treat $\mathcal{G}:\mathbb{R}^{2N} \to \mathbb{R}\cup\{+\infty\}$ as an extended real-valued function.
A prevalent algorithm to approximate the minimizer of $\mathcal{G}$, and therefore the MAP estimate $(\mathbf{x}^{\MAP},\boldsymbol{\theta}^{\MAP})$, is the so-called IAS algorithm \cite{calvetti2007gaussian,calvetti2015hierarchical,calvetti2019hierachical,calvetti2020sparse}. 
The IAS algorithm is a type of \emph{block-coordinate descent} method \cite{Bertsekas1995NonlinearP, wright2015coordinate,  beck2017first} that aims to minimize $\mathcal{G}$ by alternatingly minimizing $\mathbf{x}$ and $\boldsymbol{\theta}$. Such methods are often referred to as \emph{alternating minimization} or \emph{block Gauss-Seidel} methods.
Specifically, given an initial guess for the hyper-parameter vector $\boldsymbol{\theta}$, the IAS algorithm proceeds through a sequence of updates of the form 
\begin{equation}\label{eq:IAS}
    \boldsymbol{\theta}^{(k+1)} = \argmin_{\boldsymbol{\theta} } \left\{ \mathcal{G}(\mathbf{x}^{(k)},\boldsymbol{\theta}) \right\}, \quad 
    \mathbf{x}^{(k+1)} = \argmin_{\mathbf{x} } \left\{ \mathcal{G}(\mathbf{x},\boldsymbol{\theta}^{(k+1)}) \right\}, 
\end{equation}
until a convergence criterion is met. The IAS algorithm is motivated by the fact that the two sub-problems in \cref{eq:IAS} are easier to solve than the original optimization problem \cref{eq:IAS_MAP_estimate}.

\subsubsection*{Updating the hyper-parameters $\boldsymbol{\theta}$}

Updating the hyper-parameters $\boldsymbol{\theta}$ given $\mathbf{x}$ requires solving problems of the form $\boldsymbol{\theta}^\star = \argmin_{\boldsymbol{\theta}} \left\{ \mathcal{G}(\mathbf{x},\boldsymbol{\theta}) \right\}$. 
It was shown in \cite{calvetti2020sparse,calvetti2020sparsity} that the unique solution to the $\boldsymbol{\theta}$-update is given by
\begin{equation}\label{eq:IAS_update_beta_general} 
	\theta_i^\star = \vartheta_i \cdot \varphi\left( \frac{ | x_i | }{ \sqrt{\vartheta_i} } \right), \quad i=1,\dots,N,
\end{equation} 
where $\varphi$ is the solution to the initial value problem 
\begin{equation}\label{eq:IAS_update_beta_IVP}
	\varphi'(t) = \cfrac{2 t \varphi(t)}{ 2 r^2 \varphi(t)^{r + 1} + t^2 }, \quad \varphi(0) = (\eta/r)^{1/r},
\end{equation}
with $\eta = r \beta - 3/2$, assuming that either (i) $r < 0$ and $\eta < -\frac{3}{2}$ or (ii) $r > 0$ and $\eta > 0$. Upon ordering the components $|x_i|$ in increasing order, the updates in \cref{eq:IAS_update_beta_general} can be efficiently calculated by numerically solving \cref{eq:IAS_update_beta_IVP} only once. Moreover, for $r = \pm1$, the updates in \cref{eq:IAS_update_beta_general} admit a simple explicit solution formula  \cite{calvetti2020sparse,calvetti2020sparsity}. As a final note, since the right-hand side of  \cref{eq:IAS_update_beta_IVP} is nonnegative for $t \geq 0$, a lower bound for the optimal $\theta_i^\star$ is provided by
\begin{align} \label{eq:theta_lower}
    \theta_i^\star \geq \vartheta_i (\eta/r)^{1/r}.
\end{align}

\subsubsection*{Updating the parameter vector $\mathbf{x}$}

To update $\mathbf{x}$ given $\boldsymbol{\theta}$, we must solve $\mathbf{x}^\star = \argmin_{\mathbf{x}  } \left\{ \mathcal{G} (\mathbf{x},\boldsymbol{\theta}) \right\}$, which can be reduced to solving the quadratic optimization problem 
\begin{equation}\label{eq:IAS_x_update} 
	\mathbf{x}^\star 
		= \argmin_{\mathbf{x} \in \mathbb{R}^N} \left\{ \frac{1}{\sigma^2} \| F \mathbf{x} - \mathbf{y} \|_2^2 + \| D_{\boldsymbol{\theta}}^{-1/2} \mathbf{x} \|_2^2 \right\},
\end{equation}
where the objective is strictly convex in $\mathbf{x}$. Observe that the solution $\mathbf{x}^\star$ of \cref{eq:IAS_x_update} is the least squares solution of the overdetermined linear system
\begin{equation}\label{eq:IAS_x_update_least}
    \begin{bmatrix} \sigma^{-1} F \\ D_{\boldsymbol{\theta}}^{-1/2} \end{bmatrix} \mathbf{x} 
        = \begin{bmatrix} \sigma^{-1} \mathbf{y} \\ \mathbf{0}_N \end{bmatrix},
\end{equation}
 which in turn is the solution of  
\begin{equation}\label{eq:IAS_x_update_linear}
    \left( \sigma^{-2} F^T F + D_{\boldsymbol{\theta}}^{-1} \right) \mathbf{x} = \sigma^{-2} F^T \mathbf{y}.
\end{equation}  
Note that the coefficient matrix $\sigma^{-2} F^T F + D_{\boldsymbol{\theta}}^{-1}$ on the left-hand side of \cref{eq:IAS_x_update_linear} is symmetric and positive definite. 
Hence there is a unique solution to \cref{eq:IAS_x_update_linear} and, by extension, to the least squares problem \cref{eq:IAS_x_update_least}, as well as the quadratic optimization problem \cref{eq:IAS_x_update}. 

Various methods can be used to solve the $\mathbf{x}$-update. 
For sufficiently small $N$, employing direct methods to solve \cref{eq:IAS_x_update} at a cost of $\mathcal{O}(N^3)$ flops is reasonable. Iterative methods such as the conjugate gradient (CG) method \cite{saad2003iterative}, the conjugate gradient for least squares (CGLS) algorithm \cite{calvetti2018bayes,calvetti2020sparse} (see \cref{rem:CGLS} for more details), the gradient descent approach \cite{glaubitz2022generalized}, as well as preconditioned variants of these algorithms, are more appropriate for large $N$. 
\Cref{alg:IAS} summarizes the IAS algorithm.

\begin{algorithm}[h!]
\caption{The IAS algorithm}\label{alg:IAS} 
\begin{algorithmic}[1]
    \STATE{\textbf{Input:} Data $\mathbf{y}$, forward operator $F$, hyper-parameters $( r, \beta, \boldsymbol{\vartheta} )$, and initialization $\mathbf{x}^{(0)}$}
    \STATE{\textbf{Output:} Approximate MAP estimate $(\mathbf{x}^{\MAP},\boldsymbol{\theta}^{\MAP})$ for the joint posterior $\pi( \mathbf{x}, \boldsymbol{\theta} | \mathbf{y} )$ in \cref{eq:IAS_posterior}} 
    \REPEAT
        \STATE{Update the hyper-parameters $\boldsymbol{\theta}$ according to \cref{eq:IAS_update_beta_general}}
	\STATE{Update the parameter vector $\mathbf{x}$ according to \cref{eq:IAS_x_update}} 
    \UNTIL{convergence or the maximum number of iterations is reached}
\end{algorithmic}
\end{algorithm} 

\begin{remark}[The CGLS algorithm and priorconditioning]\label{rem:CGLS} 
The above $\mathbf{x}$-update can be obtained by direct application of the CGLS iterative method to \cref{eq:IAS_x_update_least}, with preconditioning used for enhanced efficiency. One such strategy, known as \emph{priorconditioning}, arises from making the change of variables $\mathbf{w} = D_{\boldsymbol{\theta}}^{-1/2} \mathbf{x}$ and corresponds to preconditioning by the conditional prior. In this case, \cref{eq:IAS_x_update_least} becomes 
    \begin{equation}\label{eq:IAS_x_update_least2}
        \begin{bmatrix} \sigma^{-1} F_{\boldsymbol{\theta}}  \\ I_N \end{bmatrix} \mathbf{w} 
            = \begin{bmatrix} \sigma^{-1} \mathbf{y} \\ \mathbf{0}_N \end{bmatrix}, 
    \end{equation} 
    where $F_{\boldsymbol \theta} = F D_{\boldsymbol{\theta}}^{1/2}$. Note that \cref{eq:IAS_x_update_least2} corresponds to putting the least squares problem \eqref{eq:IAS_x_update_least} into standard form \cite{Elden1977LeastSquares, Hansen2013Oblique},  and can be solved approximately for the optimal $\mathbf{w}^\star$ using the CGLS with the standard stopping criterion based on the relative residual norm. 
    The solution to the original problem is recovered as $\mathbf{x}^\star = D_{\boldsymbol{\theta}}^{1/2} \mathbf{w}^\star$. 
\end{remark}

\subsection{Existing convexity results}
\label{sub:background_convexity}

We now provide some results on the convexity of the objective function $\mathcal{G}(\mathbf{x}, \boldsymbol{\theta})$ and the convergence of the IAS algorithm in \cref{alg:IAS}. \cref{thm:previous_convexity_result} summarizes how the values of the hyper-prior parameters $r$, $\beta$, and $\vartheta$ affect the convexity properties of the objective function $\mathcal{G}$.
Originating in \cite{calvetti2020sparse}, these results were extended to multiple measurement vectors in \cite{glaubitz2024leveraging}. 

\begin{theorem}\label{thm:previous_convexity_result}
    Let $r \in \R \setminus\{ 0 \}$ and $\beta > 0$. 
    Furthermore, let $\mathcal{G}(\mathbf{x},\boldsymbol{\theta})$ be the objective function in \cref{eq:IAS_G} and let $\eta = r \beta - 3/2$. 
    \begin{enumerate}
	\item 
	If $r \geq 1$ and $\eta > 0$, then $\mathcal{G}(\mathbf{x},\boldsymbol{\theta})$ is globally strictly convex.

	\item 
	If $0 < r < 1$ and $\eta > 0$, or, if $r < 0$, then $\mathcal{G}(\mathbf{x},\boldsymbol{\theta})$ is locally convex at $(\mathbf{x}, \boldsymbol{\theta})$  provided that
	\begin{equation}
		\theta_i <  \vartheta_i \left( \cfrac{\eta}{r | r-1 |} \right)^{1/r}, \quad i=1,\ldots,N.
	\end{equation}
	
    \end{enumerate} 
	
\end{theorem}

\subsection{New convergence results}
\label{sub:background_convergence}

\cref{thm:previous_convexity_result} has implications for what can be expected from the output of \cref{alg:IAS}. In the case of a strictly convex model with $r \geq 1$ and $\eta > 0$,  $\mathcal{G}$ is globally strictly convex. 
For the subcase of $r = 1$, it was shown in \cite{calvetti2019hierachical} that \cref{alg:IAS} indeed converges to the global minimizer at a rate that is linear on the support of $\mathbf{x}$ and quadratic off the support. \cref{thm:IAS_G_convergence_convex_case} provides a new general convergence result in the strictly convex case for $r \geq 1$ and $\eta > 0$, which implies that the IAS algorithm is guaranteed to converge to the unique minimizer of $\mathcal{G}$. 

\begin{theorem}\label{thm:IAS_G_convergence_convex_case}
    Let $\mathcal{G}$ denote the objective in \cref{eq:IAS_G}, and let $\{ (\mathbf{x}^{(k)}, \boldsymbol{\theta}^{(k)}) \}$ denote the sequence of iterates of the IAS algorithm in \cref{alg:IAS}. If $r \geq 1$ and $\eta > 0$, then $\{ (\mathbf{x}^{(k)}, \boldsymbol{\theta}^{(k)}) \} \to (\mathbf{x}^{\MAP}, \boldsymbol{\theta}^{\MAP})$ as $k \to \infty$, where $(\mathbf{x}^{\MAP}, \boldsymbol{\theta}^{\MAP})$ is the unique global minimizer of $\mathcal{G}$ corresponding to \cref{eq:IAS_MAP_estimate}.
\end{theorem}
\begin{proof}
    See \cref{app:convergence_part_one}.
\end{proof}

While global convexity simplifies the MAP estimate calculation, there are compelling reasons to choose hyper-parameters $(r, \beta, \vartheta)$ that lead to nonconvex, strongly sparsity-promoting models. 
 In particular, a deviation from the global convexity of the objective function can reinforce the sparsity of the minimizer (e.g., see \cite{Wen2018SurveyNonconvex}). However, a nonconvex objective $\mathcal{G}$ may cause spurious local minima to develop, and \cref{alg:IAS} may get stuck in one of these. Since to our knowledge a convergence result for \cref{alg:IAS} in the nonconvex regime has not previously been presented, we do so in \cref{thm:IAS_G_limit_points}.

\begin{theorem}\label{thm:IAS_G_limit_points}
    Let $\mathcal{G}$ denote the objective in \cref{eq:IAS_G}, and let $\{ (\mathbf{x}^{(k)}, \boldsymbol{\theta}^{(k)}) \}$ denote the sequence of iterates of the IAS algorithm in \cref{alg:IAS}. Then $\{ (\mathbf{x}^{(k)}, \boldsymbol{\theta}^{(k)}) \}$ is bounded, and any limit point of $\{ \mathbf{x}^{(k)}, \boldsymbol{\theta}^{(k)} \}$ is a stationary point of $\mathcal{G}$.
\end{theorem}
\begin{proof}
    See \cref{app:convergence_part_one}.
\end{proof}

Note that \cref{thm:IAS_G_limit_points} is a weak convergence result in the sense that it does not guarantee that the sequence of IAS iterates tends to a stationary point of $\mathcal{G}$. 
A stronger result can be obtained by applying abstract convergence results for descent methods for nonconvex problems (e.g., see \cite{Attouch2013Convergence}). 
Specifically, suppose that the objective function is a Kurdyka-Łojasiewicz (KL) function. 
In that case, descent methods that satisfy certain sufficient decrease, relative error, and continuity conditions are guaranteed to produce stationary points of the objective. 
For example, the IAS method for the MAP estimation under a horseshoe prior developed in \cite{Dong2023Horshoe} is shown to produce a stationary point for its associated objective function using results from \cite{Attouch2013Convergence}.

To help avoid premature termination at a local but globally sub-optimal minimizer, \cite{calvetti2020sparsity} proposed hybrid versions of the IAS algorithm. These variations initially utilize the global convergence associated with gamma hyper-priors ($r=1$) to approach the vicinity of the unique minimizer before switching to a generalized gamma hyper-prior with $r < 1$ to promote greater sparsity in the solution. Following a similar philosophy, \cite{si2024path} applied path-following methods to develop a variant of the hybrid method wherein the hyper-parameters are continuously varied along a path in the hyper-parameter space.

\begin{remark}[Promoting sparsity under a linear transformation]\label{rem:transform_sparsity} 
Until now our description of hierarchical models for promoting sparsity and the IAS method for their MAP estimation has assumed the desire to promote sparsity in the parameter vector $\mathbf{x}$. 
However, oftentimes one does not wish to promote sparsity in $\mathbf{x}$ per se, but rather in some linear transformation $R \mathbf{x}$ with $R \in \mathbb{R}^{k \times n}$. Under the trivial kernel assumption that $\ker(R) = \{ \mathbf{0}_N \}$,  modifying both the hierarchical model in \Cref{sub:background_model} and its MAP estimation procedure via the IAS algorithm in \Cref{sub:background_MAP} to accommodate a sparsifying transformation $R$ is straightforward. For example, the methods in \cite{calvetti2020sparse, calvetti2020sparsity, si2024path, Uribe2022, Dong2023Horshoe} have considered taking $R$ to be a discrete-gradient operator with a zero boundary condition.  The convexity result of \cite{calvetti2020sparse} (\cref{thm:previous_convexity_result} presented in \Cref{sub:background_convexity}) no longer applies to the IAS method, however, as it was only proven for the case $R = I_N$. This also affects the convergence results. Moreover, the procedure and its analysis in the general case $\ker(R) \neq \{ \mathbf{0}_N \}$ remain mostly unexplored. Our results in \Cref{sec:main_results} serve to fill this gap.  
\end{remark}
\section{Generalized sparsity-promoting solvers}\label{sec:main_results}
The pre-existing IAS method (see \Cref{sec:background}) requires that the sparsifying transformation $R$ satisfies a trivial kernel condition $\ker(R) = \{ \mathbf{0}_N \}$ and that the signal noise variance is known \emph{a priori}. We now present a generalization of the IAS method that is able both to accommodate general sparsifying transformations $R \in \mathbb{R}^{K \times N}$ as well as  learn the noise variance $\sigma^2$ from the data.

We begin by imposing the weaker assumption that $R$ satisfies the common kernel condition 
\begin{align}\label{eq:common_kernel_condition}
    \ker(F) \cap \ker(R) = \{ \mathbf{0}_N \}
\end{align}
with respect to forward operator $F$ in \eqref{eq:intro_IP}. 
This immediately holds for $\ker(R) = \{ \mathbf{0}_N \}$ and is readily satisfied for a wide range of sparsifying transforms. In particular, our generalization to transformations $R$ with non-trivial kernels permits  various  discrete gradient operators not previously allowed, including both anisotropic and isotropic ones, across multiple dimensions. Furthermore, in satisfying this weaker assumption we avoid having to impose otherwise required artificial boundary conditions, such as Aristotelian boundary conditions \cite{Calvetti2006Aristotelian, calvetti2023bayesian}, which may be difficult to implement or are not consistent with the solution. Our generalization also enables the incorporation of other sparsifying transforms, such as discrete wavelet transformations \cite{rioul1991wavelets} and polynomial annihilation operators \cite{archibald2005polynomial,archibald2016image}. Even beyond the scope of the hierarchical models considered here, our treatment of general $R$ could be extended to other flavors of sparsity-promoting hierarchical models, such as those employing horseshoe priors \cite{Uribe2022, Dong2023Horshoe}. 

To remove the requirement that the noise variance $\sigma^2$ is known \emph{a priori}, we opt to treat the noise variance as a random variable endowed with a generalized gamma hyper-prior. 
This allows the noise variance to become a parameter that is learned during the inference procedure. 
Moreover, in addition to representing the uncertainty about the noise variance, the corresponding random variable in our new approach also encapsulates model discrepancies, which is a recurrent issue in almost all realistic scenarios. 
We also note that evidence from similar models indicates that employing a random variable can bolster results, even when the exact noise variance is known \cite{zhang2011clarify}.

\subsection{The hierarchical Bayesian model}
\label{sub:main_noise_model}

Consider the linear data model \cref{eq:intro_IP} with $\mathbf{e} \sim \mathcal{N}(\mathbf{0},\nu I)$. 
We now assume that the transformed parameter vector $R \mathbf{x}$ with $R \in \R^{K \times N}$ is sparse, and we will treat the noise variance $\nu$ as a random variable. This motivates the hierarchical model
\begin{equation}\label{eq:IAS_model_noise_variance} 
\begin{aligned}
	\mathbf{y} | \mathbf{x}, \nu 
		& \sim \mathcal{N}( F \mathbf{x}, \nu I ), \\ 
	R \mathbf{x} | \boldsymbol{\theta} 
		& \sim \mathcal{N}(\mathbf{0}_K ,D_{\boldsymbol{\theta}}), \\  
        \nu 
            & \sim  \textrm{GG}(\tilde{r}, \tilde{\beta}, \tilde{\vartheta}), \\ 
	\theta_i 
		&\overset{\text{ind}}{\sim}  \textrm{GG}(r, \beta, \vartheta_i ), \quad i=1,\ldots,K.
\end{aligned}
\end{equation}
Observe that \cref{eq:IAS_model_noise_variance} and \cref{eq:IAS_model} differ in several ways: 
(1) the likelihood density is now a conditionally Gaussian distribution, instead of a Gaussian distribution conditioned on the generalized gamma distributed variance parameter $\nu$; 
(2) the transformed parameter $R \mathbf{x}$ --rather than $\mathbf{x}$ itself-- conditioned on $\boldsymbol{\theta}$ follows a zero-mean normal distribution; 
(3) and there are now $K$ instead of $N$ hyper-parameters $\theta_i$, where $K$ is the number of rows of the sparsifying transform $R$.
We use a generalized gamma distribution for $\nu$ as this allows us to derive an update rule similar to that of $\boldsymbol{\theta}$, and because it encompasses many of the uninformative variance hyper-priors employed in the literature. As before, Bayes' theorem yields the joint posterior density function for the model parameters as $\pi( \mathbf{x}, \boldsymbol{\theta}, \nu | \mathbf{y} ) \propto \pi( \mathbf{y} | \mathbf{x}, \nu ) \pi( \mathbf{x} | \boldsymbol{\theta} ) \pi(\boldsymbol{\theta}) \pi(\nu)$.

\subsection{MAP estimation and the IAS algorithm}
\label{sub:main_noise_MAP}

Analogous to what follows \cref{eq:IAS_posterior}, the estimation for the hierarchical model  \cref{eq:IAS_model_noise_variance} corresponds to minimizing the objective function
\begin{equation}\label{eq:IAS_G_with_noise_variance} 
\begin{aligned}
    \mathcal{G}( \mathbf{x}, \boldsymbol{\theta}, \nu ) 
	= \ & \frac{1}{2 \nu} \| F \mathbf{x} - \mathbf{y} \|_2^2 
            + \frac{1}{2} \| D_{\boldsymbol{\theta}}^{-1/2} R \mathbf{x} \|_2^2 
            + \left( \frac{\nu}{\tilde{\vartheta}} \right)^{\tilde{r}}  
            + \sum_{i=1}^K \left( \frac{\theta_i}{\vartheta_i} \right)^{r} \\ 
            & - \left( \tilde{r} \tilde{\beta} - [M+2]/2 \right) \log( \nu )
            - \left( r \beta - 3/2 \right) \sum_{i=1}^K \log( \theta_i ) + \delta_{\mathbb{R}_{+}^{K+1}}(\boldsymbol{\theta}, \nu). 
\end{aligned} 
\end{equation} 

In contrast to \eqref{eq:IAS}, our generalized IAS algorithm now proceeds through a sequence of updates of the parameters $\boldsymbol{\theta}, \mathbf{x}, \nu$, until a convergence criterion is met. 
The new update rules for $\boldsymbol{\theta}$ and $\nu$ are similar to the update rule for $\boldsymbol{\theta}$ in \cref{eq:IAS}. 
The new update rule for $\mathbf{x}$ is also comparable, however, with the caveat that employing an analogue of the priorconditioning technique described in \cref{rem:CGLS} leads to a more complicated algorithm.

\subsubsection{Updating the hyper-parameters} 
\label{sec:updatehyper}

Similar to before, minimizing $\mathcal{G}$ in \cref{eq:IAS_G_with_noise_variance} for $\boldsymbol{\theta} \in \mathbb{R}_{+}^K$ with $\mathbf{x}$ and $\nu$ held fixed yields the update rule  
\begin{equation}\label{eq:IAS_update_beta_general_with_transformation} \theta_i^\star = \vartheta_i \cdot \varphi\left( \frac{ | [R \mathbf{x}]_i | }{ \sqrt{\vartheta_i} } \right), \quad i=1,\dots,K,
\end{equation}
where $\varphi$ is the solution to the initial value problem \cref{eq:IAS_update_beta_IVP} with $\eta = r \beta - 3/2$, assuming that either (i) $r < 0$ and $\eta < -\frac{3}{2}$ or (ii) $r > 0$ and $\eta > 0$. Moreover, the ODE has an analytic solution for certain cases, including  $r = \pm 1$, meaning that  \cref{eq:IAS_update_beta_general_with_transformation} can be obtained using a simple analytic formula.

\subsubsection{Updating the unknown noise variance $\nu$}\label{subsub:noise_variance}

Minimizing $\mathcal{G}$ in \cref{eq:IAS_G_with_noise_variance} for $\boldsymbol{\nu} \in \mathbb{R}_{+}$ with fixed $\mathbf{x}$ and $\boldsymbol{\theta}$ is equivalent to solving 
\begin{equation}\label{eq:IAS_update_nu}
    \nu^\star = \argmin_{\nu \geq 0} \left\{
        \frac{1}{2 \nu} \| F \mathbf{x} - \mathbf{y} \|_2^2 
        + \left( \frac{\nu}{\tilde{\vartheta}} \right)^{\tilde{r}} 
        - \tilde{\eta} \log( \nu )
    \right\}
\end{equation}
with $\tilde{\eta} = \tilde{r} \tilde{\beta} - (M+2)/2$. 
Following the arguments in \cite{calvetti2020sparse,calvetti2020sparsity}, the unique solution to \cref{eq:IAS_update_nu} is 
\begin{equation}\label{eq:IAS_update_nu2}
    \nu^\star 
        = \tilde{\varphi} \cdot \psi\left( 
            \frac{\| F \mathbf{x} - \mathbf{y} \|_2}{\tilde{\vartheta}^{1/2}} 
        \right),
\end{equation}
where $\psi$ is the solution to the initial value problem 
\begin{equation}\label{eq:IAS_update_nu_IVP}
    \psi'(t) = \frac{2 t \psi(t)}{ 2 \tilde{r}^2 \psi(t)^{\tilde{r} + 1} + t^2 }, 
    \quad \psi(0) = \left( \frac{\tilde{\eta}}{\tilde{r}} \right)^{ 1/\tilde{r} },
\end{equation}
assuming that either (i) $\tilde{r} < 0$ and $\tilde{\eta} < -(M+2)/2$ or (ii) $\tilde{r} > 0$ and $\tilde{\eta} > 0$. 
Finally, as before, \cref{eq:IAS_update_nu} becomes a quadratic problem in $\nu$ that admits a simple explicit solution formula when $r = \pm 1$.

\subsubsection{Updating the parameter vector $\mathbf{x}$}\label{sub:main_x_update}

Similarly to before, minimizing $\mathcal{G}$ in \cref{eq:IAS_G_with_noise_variance} for $\boldsymbol{\mathbf{x}} \in \mathbb{R}^N$ with fixed $\boldsymbol{\theta}$ and $\boldsymbol{\nu}$ reduces to solving the quadratic optimization problem 
\begin{equation}\label{eq:IAS_x_update_with_transformation} 
	\mathbf{x}^\star 
		= \argmin_{\mathbf{x} \in \mathbb{R}^N} \left\{ \frac{1}{\nu} \| F \mathbf{x} - \mathbf{y} \|_2^2 + \| D_{\boldsymbol{\theta}}^{-1/2} R \mathbf{x} \|_2^2 \right\},
\end{equation}
which is equivalent to solving for the least squares solution $\mathbf{x}^\star$  of the overdetermined linear system 
\begin{equation}\label{eq:IAS_x_update_least_generalized}
    \begin{bmatrix} \nu^{-1/2} F \\ D_{\boldsymbol{\theta}}^{-1/2} R \end{bmatrix} \mathbf{x} 
        = \begin{bmatrix} \nu^{-1/2} \mathbf{y} \\ \mathbf{0}_K \end{bmatrix}
\end{equation} 
or solving the regular linear system
\begin{equation}\label{eq:parameter_update_generalized}
    \left( \nu^{-1} F^T F + R^T D_{\boldsymbol{\theta}}^{-1} R \right) \mathbf{x} = \nu^{-1} F^T \mathbf{y}.
\end{equation}
Observe that the matrix $\nu^{-1} F^T F + R^T D_{\boldsymbol{\theta}}^{-1} R$ is symmetric positive-definite as long as the common kernel condition \cref{eq:common_kernel_condition} is satisfied. 
In this case, each of \cref{eq:parameter_update_generalized,eq:IAS_x_update_least_generalized,eq:IAS_x_update_with_transformation} shares the same unique solution.   \Cref{alg:IAS_noise} summarizes the resulting new {\em generalized IAS algorithm} for MAP estimation. In \Cref{sec:priorconditioning} we introduce a priorconditioning approach  designed to enhance its computational efficiency.

\begin{algorithm}[h]
\caption{The generalized IAS algorithm}\label{alg:IAS_noise} 
\begin{algorithmic}[1]
    \STATE{\textbf{Input:} Data $\mathbf{y}$, forward operator $F$, sparsifying operator $R$, hyper-parameters $( r, \beta, \vartheta )$ and $( \tilde{r}, \tilde{\beta}, \tilde{\vartheta} )$, and initialization $\mathbf{x}^{(0)}$}
    \STATE{\textbf{Output:} Approximate MAP estimate $(\mathbf{x}^{\MAP},\boldsymbol{\theta}^{\MAP}, \nu^{\MAP} )$ for the posterior of \cref{eq:IAS_model_noise_variance}} 
    \REPEAT
        \STATE{Update the hyper-parameters $\boldsymbol{\theta}$ according to \cref{eq:IAS_update_beta_general_with_transformation}}
        \STATE{Update the noise hyper-parameter $\nu$ according to \cref{eq:IAS_update_nu2}}
	\STATE{Update the parameter vector $\mathbf{x}$ according to \cref{eq:IAS_x_update_with_transformation} with $\sigma^2 = \nu$} 
    \UNTIL{convergence or the maximum number of iterations is reached}
\end{algorithmic}
\end{algorithm}

\subsection{Priorconditioning with general sparsifying transformation} \label{sec:priorconditioning}

As already mentioned, employing the priorconditioning strategy described \cref{rem:CGLS} now requires more attention paid to the sparsifying transformation $R$ if $R$ has a nontrivial kernel. 
In contrast to the priorconditioning approach detailed in \cref{rem:CGLS} for the case $R = I_N$,  the change of variables $\mathbf{w} = D_{\boldsymbol{\theta}}^{-1/2} R \mathbf{x}$ cannot be used since $\mathbf{x}$ cannot be recovered from $\mathbf{w}$ if $R$ has a non-trivial kernel. Specifically, $\mathbf{x}$ can only be recovered up to the part in the kernel (null space) of $R$. To avoid having to impose additional restrictions or constraints (see discussion following \cref{eq:common_kernel_condition}), we now introduce a new priorconditioning approach that can be applied to any general sparsifying transform satisfying the common kernel condition. 
To this end, consider the splitting of the least squares solution $\mathbf{x}^\star$ of \cref{eq:IAS_x_update_least_generalized} as 
\begin{equation}\label{eq:x_decomposition}
    \mathbf{x}^\star = \mathbf{x}_{\ker} + \mathbf{x}_{\perp},
\end{equation} 
where $\mathbf{x}_{\ker}$ is an element of the kernel of $R$, i.e., $R \mathbf{x}_{\ker} = \mathbf{0}_K$, and $\mathbf{x}_{\perp}$ is an element of the $F$-weighted orthogonal complement $\ker(R)^{\perp_F}$, i.e., $(F \mathbf{x}_{\perp})^T (F \mathbf{z}) = 0$ for all $\mathbf{z} \in \ker(R)$. This splitting makes it possible to compute $\mathbf{x}$ in a manner akin to the priorconditioning approach.

\subsubsection{An oblique projection approach}\label{sec:obliqueprojection}

In what follows, let $W \in \mathbb{R}^{N \times P}$,  $P =  \operatorname{dim}(\ker(R))$, be a  matrix with orthonormal columns such that $\operatorname{col}(W) = \ker(R)$ and let $(\cdot)^\dagger$ denote the Moore-Penrose pseudoinverse. Classical inverse problems methodology \cite[Section 8.5]{Hansen2013Oblique, hansen2010discrete} then provides the two components in \cref{eq:x_decomposition} as
\begin{align}\label{eq:x_decomposition_explicit}
    \mathbf{x}_{\ker} = W ( F W)^\dagger \mathbf{y}, \quad \mathbf{x}_{\perp} = R_{\boldsymbol{\theta}}^{\#} \mathbf{w}^\star, \quad R_{\boldsymbol{\theta}}^{\#} = \left( I_N - W(F W)^\dagger F \right) R_{\boldsymbol{\theta}}^\dagger.
\end{align}
Here $R_{\boldsymbol{\theta}}^{\#}$
 denotes the oblique ($F$-weighted) pseudoinverse of $R_{\boldsymbol{\theta}} \coloneqq D_{\boldsymbol{\theta}}^{-1/2} R$. Furthermore, $\mathbf{w}^\star$ in \cref{eq:x_decomposition_explicit} is the unique least squares solution of the ``whitened'' linear system 
\begin{equation}\label{eq:IAS_x_update_least_whitened}
    \begin{bmatrix} \nu^{-1/2} F R_{\boldsymbol{\theta}}^{\#} \\ I_K \end{bmatrix} \mathbf{w} 
        = \begin{bmatrix} \nu^{-1/2} \mathbf{y} \\ \mathbf{0}_K \end{bmatrix}.
\end{equation} 
Observe that the procedure resulting in \eqref{eq:IAS_x_update_least_whitened} corresponds to putting the least squares problem \eqref{eq:IAS_x_update_least} into standard form \cite{Elden1977LeastSquares, Hansen2013Oblique}. Therefore, to obtain the solution to \cref{eq:IAS_x_update_with_transformation}, we first solve \cref{eq:IAS_x_update_least_whitened} for $\mathbf{w}^\star$ using the CGLS method described in \cref{rem:CGLS}, and then directly  recover  $\mathbf{x}^\star$  in \eqref{eq:x_decomposition} from \cref{eq:x_decomposition_explicit}.

\begin{remark}[Special cases]
    If $R$ is invertible then $R_{\boldsymbol{\theta}}^{\#} = R^{-1}_{\boldsymbol \theta}$ and $\mathbf{x}_{\text{ker}} = \mathbf{0}_N$, and the solution $\mathbf{x}^\star$ to \cref{eq:IAS_x_update_least_generalized} is given by $\mathbf{x}^\star = R_{\boldsymbol{\theta}}^{-1} \mathbf{w}^\star$, where $\mathbf{w}^\star$ is the least squares solution to
    \begin{align}
       \begin{bmatrix} \nu^{-1/2} F R_{\boldsymbol{\theta}}^{-1}  \\ I_N \end{bmatrix} \mathbf{w} = \begin{bmatrix}
           \nu^{-1/2} \mathbf{y} \\ \mathbf{0}_N
       \end{bmatrix}.
    \end{align}
    If $R$ has full column rank with $R \in \R^{K \times N}$ then $R_{\boldsymbol{\theta}}^{\#} = R_{\boldsymbol{\theta}}^\dagger$ and $\mathbf{x}_{\text{ker}} = \mathbf{0}_N$. In this case the solution $\mathbf{x}^\star$ to \cref{eq:IAS_x_update_least_generalized} is given by $\mathbf{x}^\star = R_{\boldsymbol{\theta}}^{\dagger} \mathbf{w}^\star$, where $\mathbf{w}^\star$ is the least squares solution to 
     \begin{align}
       \begin{bmatrix} \nu^{-1/2} F R_{\boldsymbol{\theta}}^{\dagger}  \\ I_K \end{bmatrix} \mathbf{w} = \begin{bmatrix}
           \nu^{-1/2} \mathbf{y} \\ \mathbf{0}_K
       \end{bmatrix}.
    \end{align}
    Both cases represent forms of priorconditioning currently used, see \cite{calvetti2020sparse,calvetti2020sparsity} and references therein. 
    For the more general rank-deficient case, the procedure resulting in \cref{eq:IAS_x_update_least_whitened} offers a generalization for priorconditioning that has not yet been considered in the context of sparsity-promoting hierarchical Bayesian models. 
\end{remark}

\subsubsection{Efficient implementation of priorconditioning}

The computational bottleneck of solving \cref{eq:IAS_x_update_least_whitened} for large $N$ is the implementation of the potentially expensive matrix-vector product operations with the matrices $R^{\#}_{\boldsymbol{\theta}}$ and $(R^{\#}_{\boldsymbol{\theta}})^T$ as expressed in \cref{eq:x_decomposition_explicit}. 
It is reasonable to assume that the dimension of the kernel of $R$, denoted by $P = \operatorname{dim}(\ker(R))$, is relatively small. 
Since $W \in \mathbb{R}^{N \times P}$, we can readily compute the economic QR decomposition $FW = \tilde{Q} \tilde{R}$ with $\tilde{Q} \in \mathbb{R}^{ M \times P }$ and $\tilde{R} \in \mathbb{R}^{ P \times P }$.
Furthermore, using $(FW)^\dagger = \tilde{R}^{-1} \tilde{Q}^T$ provides an efficient way to compute matrix-vector products involving $(FW)^{\dagger}$ or $((FW)^{\dagger})^T$. 
    
Since $\ker(R)$ and $W$ are independent of $\boldsymbol{\theta}$, we need only to compute the QR decomposition of $F W$ once and then reuse this factorization to compute matrix-vector products involving $R^{\#}_{\boldsymbol{\theta}}$ and $(R^{\#}_{\boldsymbol{\theta}})^T$ for varying $\boldsymbol{\theta}$. 
However, greater care must be taken when computing matrix-vector products with pseudoinverses $R_{\boldsymbol{\theta}}^\dagger$ or  $(R_{\boldsymbol{\theta}}^\dagger)^T$, since the QR approach will be computationally infeasible for moderate to large values of $K$. 
One approach is to instead use an approximate approach based on the relations
\begin{align}\label{eq:approx_pseudoinverse} 
    R_{\boldsymbol{\theta}}^\dagger \approx (R_{\boldsymbol{\theta}}^T R_{\boldsymbol{\theta}} + \delta I_N )^{-1} R_{\boldsymbol{\theta}}^T, \quad
    (R_{\boldsymbol{\theta}}^\dagger)^T \approx R_{\boldsymbol{\theta}} (R_{\boldsymbol{\theta}}^T R_{\boldsymbol{\theta}} + \delta I_N )^{-1}, 
\end{align}
for some small $\delta > 0$. This approach is particularly attractive when $R_{\boldsymbol{\theta}}^T R_{\boldsymbol{\theta}}$ is a banded matrix, in which case computing the (banded) Cholesky factorization of $R_{\boldsymbol{\theta}}^T R_{\boldsymbol{\theta}} + \delta I_N$ to apply the inverse in \cref{eq:approx_pseudoinverse} costs $\mathcal{O}(p^2 N)$ flops where $p$ denotes the bandwidth of $R_{\boldsymbol{\theta}}^T R_{\boldsymbol{\theta}}$ \cite{rue2005gaussian}. 
We note that other approaches, such as multigrid methods \cite{Briggs2000Multigrid}, are not particularly advantageous here due to the dependence of the matrix on $\boldsymbol{\theta}$ which changes across iterations.

For extremely large problems where a (banded) Cholesky factorization is infeasible, one may instead evaluate the pseudoinverse approximations in \cref{eq:approx_pseudoinverse} using a CG or CGLS method. Alternatively, and perhaps more naturally, a modified CG algorithm \cite{hestenes1975pseudoinversus, hayami2018convergence} directly computes the pseudoinverses without requiring an approximation such as \cref{eq:approx_pseudoinverse}. Specifically, the matrix-vector product $R_{\boldsymbol{\theta}}^\dagger \mathbf{v}$ is approximated by applying the CG method to the solution of $R_{\boldsymbol{\theta}}^T R_{\boldsymbol{\theta}} \mathbf{z} = R_{\boldsymbol{\theta}}^T \mathbf{v}$ for $\mathbf{z} \in \operatorname{col}(R^T)$, initialized by some $\mathbf{z}^{(0)} \in \operatorname{col}(R^T)$. The method may be further accelerated with a preconditioner. 
For example, suppose $R$ represents a two-dimensional anisotropic discrete gradient operator with Neumann boundary conditions. In this case one may take advantage of a fast DCT-based spectral preconditioner derived from the unweighted discretized Laplacian $R^T R$ \cite{hansen2006deblurring,Strang1999DCT, ghiglia1994robust}. For ease of presentation, we defer a detailed explanation of this technique to  \cref{app:iterative_pinv_computations}. 
To our knowledge, this current investigation is the first to demonstrate the use of such a preconditioning strategy in the context of implementing priorconditioning for sparsity-promoting hierarchical Bayesian models.

\subsection{Convexity and convergence} 
\label{sub:main_noise_convexity} We now provide some results on the convexity of the objective $\mathcal{G}(\mathbf{x}, \boldsymbol{\theta})$ and the convergence of the generalized IAS algorithm in \cref{alg:IAS_noise}.

\subsubsection{Convexity}\label{sub:main_convexity}

\cref{thm:main_convexity} summarizes how the values of the hyper-prior parameters $r/\tilde{r}$, $\beta/\tilde{\beta}$, and $\vartheta/\tilde{\vartheta}$ affect the convexity properties of the objective function $\mathcal{G}$.

\begin{theorem}\label{thm:main_convexity}
    Let $\mathcal{G}( \mathbf{x}, \boldsymbol{\theta}, \nu )$ be the objective function in \cref{eq:IAS_G_with_noise_variance},  and let $\eta = r \beta - 3/2$,  $\tilde{\eta} = \tilde{r} \tilde{\beta} - [M+2]/2$. 
    Let $R \in \mathbb{R}^{K \times N}$ satisfy the common kernel condition \cref{eq:common_kernel_condition}.
    \begin{enumerate}
        \item 
	  If $r, \tilde{r} \geq 1$ and $\eta, \tilde{\eta} > 0$, then $\mathcal{G}( \mathbf{x}, \boldsymbol{\theta}, \nu )$ is globally strictly convex. 

	  \item 
        Assume the following two conditions hold: 
        \begin{enumerate}
            \item[(a)] \{ $0 < r < 1$ and $\eta > 0$ \} or \{ $r < 0$ \};
            \item[(b)] \{ $0 < \tilde{r} < 1$ and $\tilde{\eta} > 0$ \} or \{ $\tilde{r} < 0$  \}.
        \end{enumerate}
        Then $\mathcal{G}( \mathbf{x}, \boldsymbol{\theta}, \nu )$ is locally convex at $(\mathbf{x}, \boldsymbol{\theta}, \nu)$ provided that
        \begin{equation} 
            \theta_i < \vartheta_i \left( \frac{\eta}{r | r - 1 |} \right)^{1/r} 
            \quad \text{and} \quad 
            \nu < \tilde{\vartheta} \left( \frac{\tilde{\eta}}{\tilde{r} | \tilde{r} - 1 |} \right)^{\frac{1}{\tilde{r}}}, \quad i = 1,\dots,K.
	  \end{equation} 
    \end{enumerate} 	
\end{theorem}
\begin{proof}
    See \cref{app:main_convexity_convergence}.
\end{proof}

\cref{thm:main_convexity} demonstrates that when the noise variance is treated as a generalized gamma-distributed random variable, new conditions arise for the convexity of the corresponding objective function $\mathcal{G}$.
The condition $\tilde{\eta} > 0$ is particularly stringent if global strict convexity is desired since it requires that $\tilde{r} \tilde{\beta} > \frac{M+2}{2}$, and in practice implies that $\tilde{r} \tilde{\beta}$ must be significantly large for a large number of observations. One variance hyper-prior that is frequently used in the literature (e.g., see \cite{Bardsley2012mcmcimaging, Fox2016FastSampling, Saibaba2019marginalization, Brown2018lowrank, gelman2006variance, glaubitz2022generalized}) is an ``uninformative'' inverse gamma distribution,\footnote{Note that an inverse gamma hyper-prior for a variance parameter is equivalent to a gamma hyper-prior for a precision (reciprocal of variance) parameter, as used in \cite{Bardsley2012mcmcimaging, Fox2016FastSampling, Saibaba2019marginalization, glaubitz2022generalized}.} which corresponds to the choice of hyper-prior parameters $\tilde{r} = -1, \tilde{\beta} = 1, \tilde{\vartheta} \approx 0$. 
This type of prior is meant to be weakly informative for the noise variance $\nu$. It approximates the Jeffreys prior $\pi(\nu) \propto \nu^{-1}$, which is the unique (but improper) prior distribution that is uninformative of the scale of $\nu$ \cite{Jeffreys1946Invariant}. \Cref{thm:main_convexity} shows that the resulting objective $\mathcal{G}$ is generally nonconvex for the choices of hyper-prior parameters needed to impose an uninformative variance hyper-prior on $\nu$.

In addition to providing conditions for the convexity of \cref{eq:IAS_model_noise_variance} with the noise variance treated as a random variable,  \cref{thm:main_convexity} yields a more general convexity result for the original model \cref{eq:IAS_model} where the noise variance is assumed to be known. 
Specifically, the existing convexity result in \cref{thm:previous_convexity_result} cannot be applied to determine the convexity of \cref{eq:IAS_model} when it has been modified to include a sparsifying transformation $R$. 

\begin{corollary}\label{lemma:transformation_convexity_result}
    Let $\mathcal{G}(\mathbf{x}, \boldsymbol{\theta}; \nu)$ be the objective function in \cref{eq:IAS_G_with_noise_variance} for any fixed noise variance $\nu > 0$, and let $\eta = r \beta - 3/2$. Let $R \in \mathbb{R}^{K \times N}$ satisfy the common kernel condition \cref{eq:common_kernel_condition}. 
    \begin{enumerate}
	   \item 
	   If $r \geq 1$ and $\eta > 0$, then $\mathcal{G}(\mathbf{x},\boldsymbol{\theta}; \nu)$ is globally strictly convex in the variables $(\mathbf{x}, \boldsymbol{\theta})$.

	   \item 
	   If $0 < r < 1$ and $\eta > 0$, or, if $r < 0$, then $\mathcal{G}(\mathbf{x},\boldsymbol{\theta}; \nu)$ is locally convex in $(\mathbf{x}, \boldsymbol{\theta})$ provided that
	   \begin{equation}
		  \theta_i < \vartheta_i \left( \cfrac{\eta}{r | r-1 |} \right)^{1/r}, \quad i=1,\ldots,K.
	   \end{equation}
	
    \end{enumerate} 
	
\end{corollary}

Observe that the global and local convexity conditions in \cref{lemma:transformation_convexity_result} are equivalent to those in \cref{thm:previous_convexity_result}. 
That is, the convexity of the objective function does not change with the inclusion of a sparsifying transformation so long as the hyper-prior parameters remain unchanged. Although not part of the discussion in this investigation, the convexity result in \cref{lemma:transformation_convexity_result} allows us to amalgamate our generalized IAS algorithm with hybrid strategies, as elucidated in \cite{calvetti2020sparsity,si2024path}.

\subsubsection{Convergence}\label{sub:main_convergence} 

In light of \cref{thm:main_convexity}, we can state two convergence results for our generalized IAS algorithm in \cref{alg:IAS_noise}. The first result applies to the case of a convex model (when $r/\tilde{r} \geq 1$ and $\eta/\tilde{\eta} > 0$), in which case we see that \cref{alg:IAS_noise} is guaranteed to produce the unique global minimizer of $\mathcal{G}$. The second result states that, in the nonconvex case, any limit point of the iterates produced by \cref{alg:IAS_noise} must be a stationary point of $\mathcal{G}$. Since they closely follow those of  \cref{thm:IAS_G_convergence_convex_case,thm:IAS_G_limit_points}, the corresponding proofs are omitted.

\begin{theorem}\label{thm:main_convergence_convex_case}
    Let $\mathcal{G}$ denote the objective in \cref{eq:IAS_G_with_noise_variance} and let $\{ (\mathbf{x}^{(k)}, \boldsymbol{\theta}^{(k)}, \nu^{(k)}) \}$ denote the sequence of iterates of the  generalized IAS algorithm in \cref{alg:IAS_noise}. 
    If $r, \tilde{r} \geq 1$ and $\eta,\tilde{\eta} > 0$, then $\{ (\mathbf{x}^{(k)}, \boldsymbol{\theta}^{(k)}, \nu^{(k)} ) \} \to (\mathbf{x}^{\MAP}, \boldsymbol{\theta}^{\MAP}, \nu^{\MAP})$ as $k \to \infty$, where $(\mathbf{x}^{\MAP}, \boldsymbol{\theta}^{\MAP}, \nu^{\MAP})$ is the unique global minimizer of $\mathcal{G}$. 
\end{theorem}

\begin{theorem}\label{thm:main_convergence_nonconvex_case}
    Let $\mathcal{G}$ denote the objective in \cref{eq:IAS_G_with_noise_variance} and let $\{ (\mathbf{x}^{(k)}, \boldsymbol{\theta}^{(k)}, \nu^{(k)} ) \}$ denote the sequence of iterates of the generalized IAS algorithm in \cref{alg:IAS_noise}. 
    Then $\{ (\mathbf{x}^{(k)}, \boldsymbol{\theta}^{(k)}, \nu^{(k)} ) \}$ is bounded and any limit point of $\{ (\mathbf{x}^{(k)}, \boldsymbol{\theta}^{(k)}, \nu^{(k)} ) \}$ is a stationary point of $\mathcal{G}$.
\end{theorem}

Although not explicitly proven here, it is also possible to show global convergence of the iterates of \cref{alg:IAS_noise} to a stationary point of $\mathcal{G}$ using abstract convergence results for descent methods  \cite{Attouch2013Convergence}. Indeed, such behavior is observed in our numerical tests in \Cref{sec:numerics}.

\section{Numerical tests}
\label{sec:numerics} 

We now demonstrate the efficacy of our generalized IAS algorithm in \cref{alg:IAS_noise} through two numerical tests. The hyper-prior parameter vector $\boldsymbol{\vartheta}$ is assumed to be a constant vector, i.e., $\boldsymbol{\vartheta} =  [\vartheta_1, \ldots, \vartheta_K] = \vartheta \mathbf{1}_{K}$ for a scalar parameter $\vartheta \in \mathbb{R}_+$ which plays the role of a tunable regularization parameter. We consider various parameter sets for the remaining hyper-prior parameters $r, \beta, \tilde{r}$, $\tilde{\beta},$ and $\tilde{\vartheta}$. For a discussion of the sensitivity of the IAS method to the parameters $r$ and $\beta$, we refer the reader to \cite{calvetti2020sparse, si2024path}. Our choices of hyper-prior parameters $\tilde{r}$, $\tilde{\beta}$ $\tilde{\vartheta}$ yield an uninformative noise variance hyper-prior as discussed in \Cref{sub:main_convexity}.

\subsection{Signal denoising with different sparsifying transforms}
\label{sub:signal_denoising} 

We first apply our generalized IAS algorithm in \cref{alg:IAS_noise} to a simple signal denoising task and study the effect of varying the sparsifying transformation. We generate a ground truth vector $\overline{\mathbf{x}} \in \mathbb{R}^{N}$ by evaluating the function
\begin{align}
    f(x) =
    \begin{cases} 2 \sin(50 \pi x)  + 25 x, & 0 \leq x < 0.4, \\
2 \sin(50 \pi x)  + 25 x + 50, & 0.4 \leq x < 0.7, \\
2 \sin(50 \pi x)  + 25 x + 120, & 0.7 \leq x \leq 1,
\end{cases}
\end{align}
at $N = 1000$ equispaced points in the interval $[0, 1]$. We then generate a synthetic observation $\mathbf{y} = \overline{\mathbf{x}} + \mathcal{N}(\mathbf{0}_N, \overline{\nu} I_N )$, with $\overline{\nu} = 10$. 
The synthetic data are shown in \cref{fig:experiment_1_toy_data}.

\begin{figure}[h!]
    \centering
    \includegraphics[width=.85\textwidth]{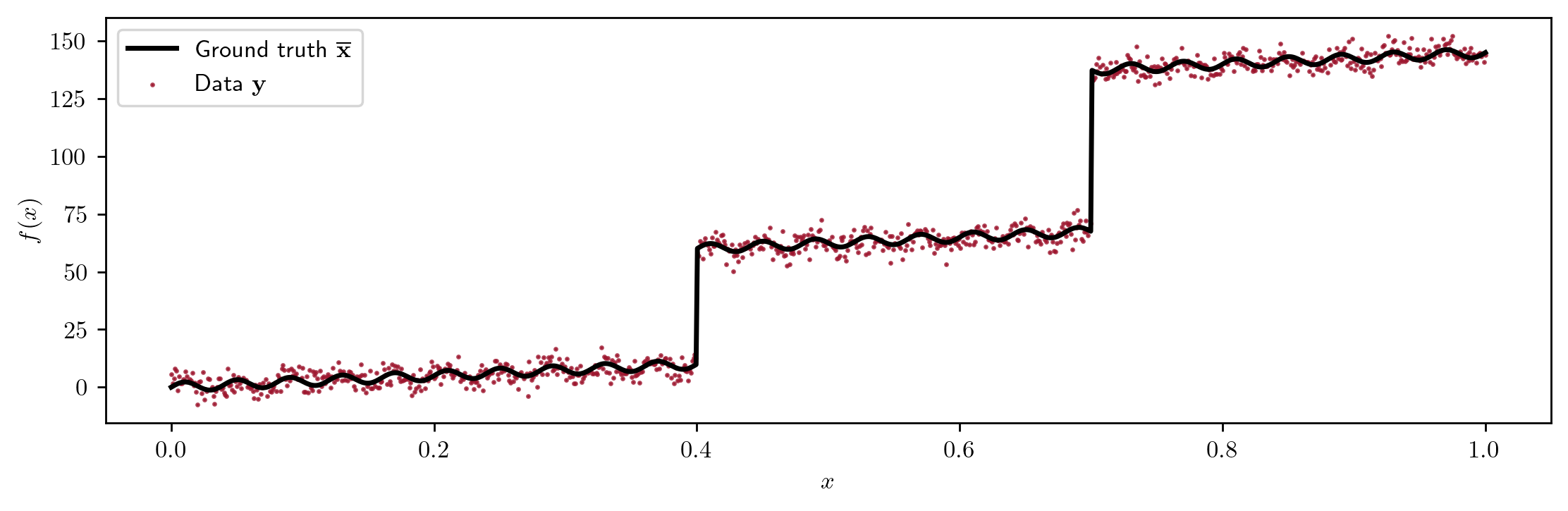}
    \caption{Ground truth signal $\overline{\mathbf{x}}$ and synthetic data $\mathbf{y}$.}
    \label{fig:experiment_1_toy_data}
\end{figure}

We seek to estimate $\overline{\mathbf{x}}$ from $\mathbf{y}$ with varying sparsity assumptions according to the three linear transformations
\begin{equation}
\label{eq:experiment_1_Rs}
\begin{aligned}
    R_1 = 
    \begin{bmatrix}
        -1 & 1  &   &  \\
        & \ddots & \ddots &  \\
        &  & -1 &  1 
    \end{bmatrix}, \ 
    R_2 = 
    \begin{bmatrix}
        -1 & 2 & -1  &   &  \\
        & \ddots & \ddots & \ddots & \\
        &  & -1 &  2 & 1
    \end{bmatrix}, \
    R_3 = 
    \begin{bmatrix}
        -1 & 3 & -3  & 1  & &  \\
        & \ddots & \ddots & \ddots & \ddots  &  \\
    &  & -1 &  3 & -3 & 1
    \end{bmatrix}.
\end{aligned}
\end{equation}
These matrices represent discretizations of the first, second, and third derivatives, respectively. Notably, $\ker(R_i) \neq \{ \mathbf{0}_N \}$ for $i \in \{ 1, 2, 3 \}$.

\begin{remark}
Observe that the kernels of $R_1$, $R_2$, and $R_3$ respectively consist of all constant signals ($R_1$), all constant and linear signals ($R_2$), and all constant, linear, and quadratic signals ($R_3$). 
It follows that $\operatorname{dim}(\operatorname{ker}(R_1)) = 1$, $\operatorname{dim}(\operatorname{ker}(R_2)) = 2$, and $\operatorname{dim}(\operatorname{ker}(R_3)) = 3$. 
That is, the linear transformations have an increasing---and non-trivial---kernel.
\end{remark}

Recall that $\{W_i\}_{i = 1}^3$, where $\ker(R_i) = \operatorname{col}(W_i)$, is needed to accelerate \cref{alg:IAS_noise} using the priorconditioning technique described in \cref{sec:priorconditioning}.
To this end, it suffices to take
\begin{align}
    W_1 = \begin{bmatrix}
        1 \\ \vdots \\ 1
    \end{bmatrix}, \quad W_2 = \begin{bmatrix}
        1 & 1 \\ 1 & 2 \\ \vdots & \vdots \\ 1 & N-1 \\ 1 & N
    \end{bmatrix}, \quad W_3 = \begin{bmatrix}
        1 & 1 & 1 \\ 1 & 2 & 3 \\ 
        1 & 3 & 6 \\ \vdots & \vdots & \vdots  \\ 
     1 & N & N(N+1)/2
    \end{bmatrix}.
\end{align}
 Let $R_{i, \boldsymbol{\theta}} = D_{\boldsymbol{\theta}}^{-1/2} R_i$. To implement the pseuodinverses $R_{i, \boldsymbol{\theta}}^\dagger$ needed for priorconditioning, we use \cref{eq:approx_pseudoinverse} to obtain
\begin{align}\label{eq:Ri_pinv}
    R_{i, \boldsymbol{\theta}}^\dagger \approx (R_{i, \boldsymbol{\theta}}^T R_{i, \boldsymbol{\theta}} + \delta I_N )^{-1} R_{i, \boldsymbol{\theta}}^T.
\end{align}
For this experiment, we set $\delta = \delta_{\text{PINV}} = 10^{-8}$ and directly employ Cholesky factorizations to compute $(R_{i, \boldsymbol{\theta}}^T R_{i, \boldsymbol{\theta}} + \delta I_N )^{-1}$.
Due to the bandwidths $p_i$ of the matrices $R_{i, \boldsymbol{\theta}}^T R_{i, \boldsymbol{\theta}} + \delta I_N$ being small (here $p_i = i$), this has a computational cost of $\mathcal{O}(p_i^2 N)$ flops. 

\setlength{\extrarowheight}{3pt}
\begin{table}[h!]
\centering
\begin{tabular}{|l|l|l|l|l|l|l|l|l|}
\hline
\textbf{Parameter} & $r$  & $\beta$  & $\tilde{r}$  & $\tilde{\beta}$  & $\tilde{\vartheta}$ & $\delta_{\text{PINV}}$  & $\varepsilon_{\text{CGLS}}$ & $\varepsilon_{\text{IAS}}$    \\ \hline
\textbf{Value}  & $1$  & $1.5 + 10^{-3}$ & $-1$ & $1$  & $10^{-4}$  & $10^{-8}$ & $10^{-4}$  & $10^{-3}$   \\ \hline
\end{tabular}\caption{Parameter values used for the first numerical test. }\label{table:experiment_1_parameters}
\end{table}

The other free parameters for \cref{alg:IAS_noise} are provided in \cref{table:experiment_1_parameters}, except for $\vartheta$. Following \cite{calvetti2019hierachical}, we choose the hyper-prior parameters $r$ and $\beta$ such that the MAP estimate is close to the $\ell_1$ regularized solution, promoting sparsity under the transform $R_i$.
Moreover, the hyper-prior parameters $\tilde{r}$, $\tilde{\beta}$, and $\tilde{\vartheta}$ impose an uninformative prior for the unknown variance $\nu$.

For the (inner loop) CGLS method performing the $\mathbf{x}$-update, we equip the CGLS method with the usual stopping criterion based on the relative residual norm with $\varepsilon_{\text{CGLS}} = 10^{-4}$. This is a relatively stringent tolerance which we utilize to standardize our comparison of \cref{alg:IAS_noise} with and without priorconditioning.  It is common in practice to solve subproblems of a coordinate descent method with less stringent tolerances. For the (outer loop) generalized IAS stopping criterion, we terminate the algorithm at the first index $k^\star$ such that the iterates satisfy both
\begin{align}\label{eq:IAS_stopping_criterion}
    \cfrac{ \| \boldsymbol{\theta}^{(k^\star) } - \boldsymbol{\theta}^{(k^\star - 1) }  \|_2  }{ \| \boldsymbol{\theta}^{(k^\star - 1)} \|_2  } < \varepsilon_{\text{IAS}} \quad \text{ and } \quad \frac{| \nu^{(k^\star)} - \nu^{(k^\star - 1)} |}{ \nu^{(k^\star - 1)} } < \varepsilon_{\text{IAS}}.
\end{align} 
Observe that by \cref{thm:main_convexity}, the parameter values in \cref{table:experiment_1_parameters} correspond to the minimization of a nonconvex objective function. This means that some care must be taken with the initialization of \cref{alg:IAS_noise}  to avoid suboptimal local minima. 
Indeed, in our numerical tests we occasionally observe that \cref{alg:IAS_noise} may learn a noise variance parameter much larger than expected if initialized with a simple initialization such as $\mathbf{x}^{(0)} = \mathbf{0}_N$. This behavior can be mitigated by using the initialization
\begin{align}\label{eq:tikhonov_initialization}
    \mathbf{x}^{(0)} = \argmin_{\mathbf{x} \in \mathbb{R}^N} \,\, \| F \mathbf{x} - \mathbf{y} \|_2^2 + \lambda \| R_i \mathbf{x} \|_2^2
\end{align}
for some suitable  $\lambda > 0$, corresponding to standard Tikhonov regularization.

\subsubsection{Reconstruction results for varying $\vartheta$}

We proceed by studying the behavior of \cref{alg:IAS_noise} applied to this signal denoising task, as a function of the remaining (regularization) parameter $\vartheta$. 
The parameter $\vartheta$ can be viewed as a tuning parameter governing the overall strength of the sparsity-promoting regularization in which small values of $\vartheta$ result in stronger regularization. 
\cref{table:experiment_one_table_R1,table:experiment_one_table_R2,table:experiment_one_table_R3} report the impact of $\vartheta$ on the solution produced by \cref{alg:IAS_noise} for each of the sparsifying transformations in \cref{eq:experiment_1_Rs}. The features we consider are: (1) the learned noise variance $\hat{\nu}$; (2) the  total number of CGLS iterations ($n_{\text{CGLS}})$ or priorconditioned CGLS ($n_{\text{PCGLS}}$) iterations expended across all IAS iterations until convergence; (3) the ratio $t_{\text{PCGLS}}/t_{\text{CGLS}}$ where $t_{\text{PCGLS}}$ and $t_{\text{CGLS}}$ denote the total wall-clock time required by the IAS algorithm with and without priorconditioning, respectively; (4) reconstruction performance metrics such as the relative reconstruction error (RRE), defined as $\text{RRE}(\mathbf{x}, \overline{\mathbf{x}}) = \| \mathbf{x} - \overline{\mathbf{x}} \|_2 / \| \overline{\mathbf{x}} \|_2$, and the structural similarity index measure (SSIM) \cite{wang2004image}; and (5) the value of the discrepancy principle residual $\operatorname{DP}(\mathbf{x},\nu) = \| F \mathbf{x} - \mathbf{y} \|_2^2 - \eta \nu M$ for a safeguard factor $\tau = 1.01$,
which measures how well the solution $\mathbf{x}$ agrees with the learned noise variance \cite{hansen2010discrete, vogel2002computational, engl1996regularization}.\footnote{For the ground truth $\overline{\mathbf{x}}$ and $\overline{\nu}$, we have $F \overline{\mathbf{x}} - \mathbf{y} \sim \mathcal{N}(\mathbf{0}_M, \overline{\nu} I_M)$, which equating the squared norms in expectation yields $\| F \overline{\mathbf{x}} - \mathbf{y} \|_2^2 = \overline{\nu} M$. Thus, we expect that $\operatorname{DP}(\mathbf{x}, \nu)$ is near zero if $(\mathbf{x}, \nu)$ is near $(\overline{\mathbf{x}}, \overline{\nu})$.}

\begin{table}[h!]
\centering
\begin{adjustbox}{width=0.6\textwidth}
\begin{tabular}{lrrrrrrr}
\toprule
$\vartheta$ & $\hat{\nu}$ & $n_{\text{CGLS}}$ & $n_{\text{PCGLS}}$ & $\frac{t_{\text{PCGLS}}}{t_{\text{CGLS}}}$ & RRE & SSIM & DP \\
\midrule
$10^{-3}$ & $16.42$ & $77887$ & $359$ & $0.13$ & 3.16\% & $0.918$ & $67.83$ \\
$10^{-2}$ & $11.35$ & $62749$ & $237$ & $0.21$ & 1.79\% & $0.946$ & $-73.16$ \\
$10^{-1}$ & $10.73$ & $17780$ & $228$ & $0.19$ & 1.64\% & $0.956$ & $-66.00$ \\
$10^{0}$ & $9.15$ & $7104$ & $495$ & $0.41$ & 1.46\% & $0.966$ & $89.38$ \\
$10^{1}$ & $0.29$ & $3117$ & $9016$ & $6.50$ & 3.51\% & $0.731$ & $-251.03$ \\
$10^{2}$ & $0.85$ & $170$ & $3481$ & $10.12$ & 3.52\% & $0.729$ & $-817.59$ \\
$10^{3}$ & $4.67$ & $25$ & $1598$ & $14.94$ & 3.41\% & $0.743$ & $-4613.07$ \\
\bottomrule
\end{tabular}
\end{adjustbox}
\caption{Summary results for the impact of varying $\vartheta$ on the output of \cref{alg:IAS_noise}, using $R_1$ in \cref{eq:experiment_1_Rs} as the sparsifying transformation.}\label{table:experiment_one_table_R1}
\end{table}
\begin{table}[h!]
\centering
\begin{adjustbox}{width=0.6\textwidth}
\begin{tabular}{lrrrrrrr}
\toprule
$\vartheta$ & $\hat{\nu}$ & $n_{\text{CGLS}}$ & $n_{\text{PCGLS}}$ & $\frac{t_{\text{PCGLS}}}{t_{\text{CGLS}}}$ & RRE & SSIM & DP \\
\midrule
$10^{-3}$ & $29.98$ & $539000$ & $2022$ & $0.05$ & 5.38\% & $0.926$ & $-158.53$ \\
$10^{-2}$ & $18.59$ & $475438$ & $2775$ & $0.04$ & 3.81\% & $0.938$ & $-111.35$ \\
$10^{-1}$ & $12.67$ & $86371$ & $5753$ & $0.26$ & 2.76\% & $0.959$ & $-73.77$ \\
$10^{0}$ & $9.74$ & $15937$ & $5833$ & $0.90$ & 2.07\% & $0.970$ & $-55.29$ \\
$10^{1}$ & $6.71$ & $2526$ & $8974$ & $5.71$ & 1.77\% & $0.946$ & $25.83$ \\
$10^{2}$ & $0.29$ & $837$ & $60562$ & $108.59$ & 3.58\% & $0.722$ & $-275.50$ \\
$10^{3}$ & $0.57$ & $108$ & $31878$ & $315.09$ & 3.62\% & $0.717$ & $-569.01$ \\
\bottomrule
\end{tabular}
\end{adjustbox}
\caption{Summary results for the impact of varying $\vartheta$ on the output of \cref{alg:IAS_noise}, using $R_2$ in \cref{eq:experiment_1_Rs} as the sparsifying transformation.}\label{table:experiment_one_table_R2}
\end{table}

\begin{table}[h!]
\centering
\begin{adjustbox}{width=0.6\textwidth}
\begin{tabular}{lrrrrrrr}
\toprule
$\vartheta$ & $\hat{\nu}$ & $n_{\text{CGLS}}$ & $n_{\text{PCGLS}}$ & $\frac{t_{\text{PCGLS}}}{t_{\text{CGLS}}}$ & RRE & SSIM & DP \\
\midrule
$10^{-3}$ & $18.53$ & $584000$ & $8308$ & $0.05$ & 3.88\% & $0.932$ & $-96.22$ \\
$10^{-2}$ & $13.22$ & $135000$ & $15084$ & $0.32$ & 2.94\% & $0.964$ & $-77.09$ \\
$10^{-1}$ & $11.14$ & $80661$ & $16027$ & $0.50$ & 2.48\% & $0.968$ & $-66.62$ \\
$10^{0}$ & $9.61$ & $33743$ & $22810$ & $1.38$ & 2.17\% & $0.966$ & $-55.79$ \\
$10^{1}$ & $7.71$ & $5724$ & $23681$ & $8.00$ & 2.06\% & $0.944$ & $-38.37$ \\
$10^{2}$ & $0.13$ & $5322$ & $71309$ & $22.14$ & 3.60\% & $0.720$ & $-118.33$ \\
$10^{3}$ & $0.24$ & $382$ & $40381$ & $58.20$ & 3.63\% & $0.716$ & $-240.97$ \\
\bottomrule
\end{tabular}
\end{adjustbox}
\caption{Summary results for the impact of varying $\vartheta$ on the output of \cref{alg:IAS_noise}, using $R_3$ in \cref{eq:experiment_1_Rs} as the sparsifying transformation.}\label{table:experiment_one_table_R3}
\end{table}

Recall that the true noise variance used to generate the data was $\overline{\nu} = 10$. We observe in all cases that the noise variance is underestimated when only little regularization ($\vartheta \approx 10^3$) is applied. On the other hand, the noise variance is overestimated when too much regularization ($\vartheta = 10^{-3}$) is applied. This observation may be explained as follows: 
As we increase regularization (promote sparsity more strongly), the model tends to explain data misfits by an increasing noise variance. 
In between these two extremes, we observe that there is a region where the learned noise variance is near the truth $\overline{\nu}$. 
Furthermore, we observe that the parameter $\vartheta$ has a significant impact on the total number of CGLS or PCGLS iterations. 
In the under-regularized regime ($\vartheta = 10^{3}$) we see that $n_{\text{PCGLS}}$ is extremely small compared to $n_{\text{CGLS}}$, whereas the reverse holds in the over-regularized regime ($\vartheta = 10^{-3}$). 
Thus, whether priorconditioning provides a reduction in the number of CGLS iterations is highly dependent on the strength of the regularization imposed. Comparing the reductions in wall-clock time gained by priorconditioning, we observe that the potential benefit of priorconditioning is not nearly as pronounced as it appears in the reduction of CGLS iterations. However, we anticipate that this ratio will depend dramatically on the cost of performing matrix-vector products with the measurement operator $F$. This experiment presents the least favorable conditions for priorconditioning in a wall-clock time comparison, since the forward operator $F$ being equal to the identity does not present any cost when implemented as a function handle. For $F$ of increasing complexity, we expect a wall-clock time comparison to appear increasingly advantageous for priorconditioning when $n_{\text{PCGLS}} < n_{\text{CGLS}}$.

\subsubsection{Reconstruction results with optimal $\vartheta$}

\begin{figure}[h!]
    \centering
    \includegraphics[width=.85\textwidth]{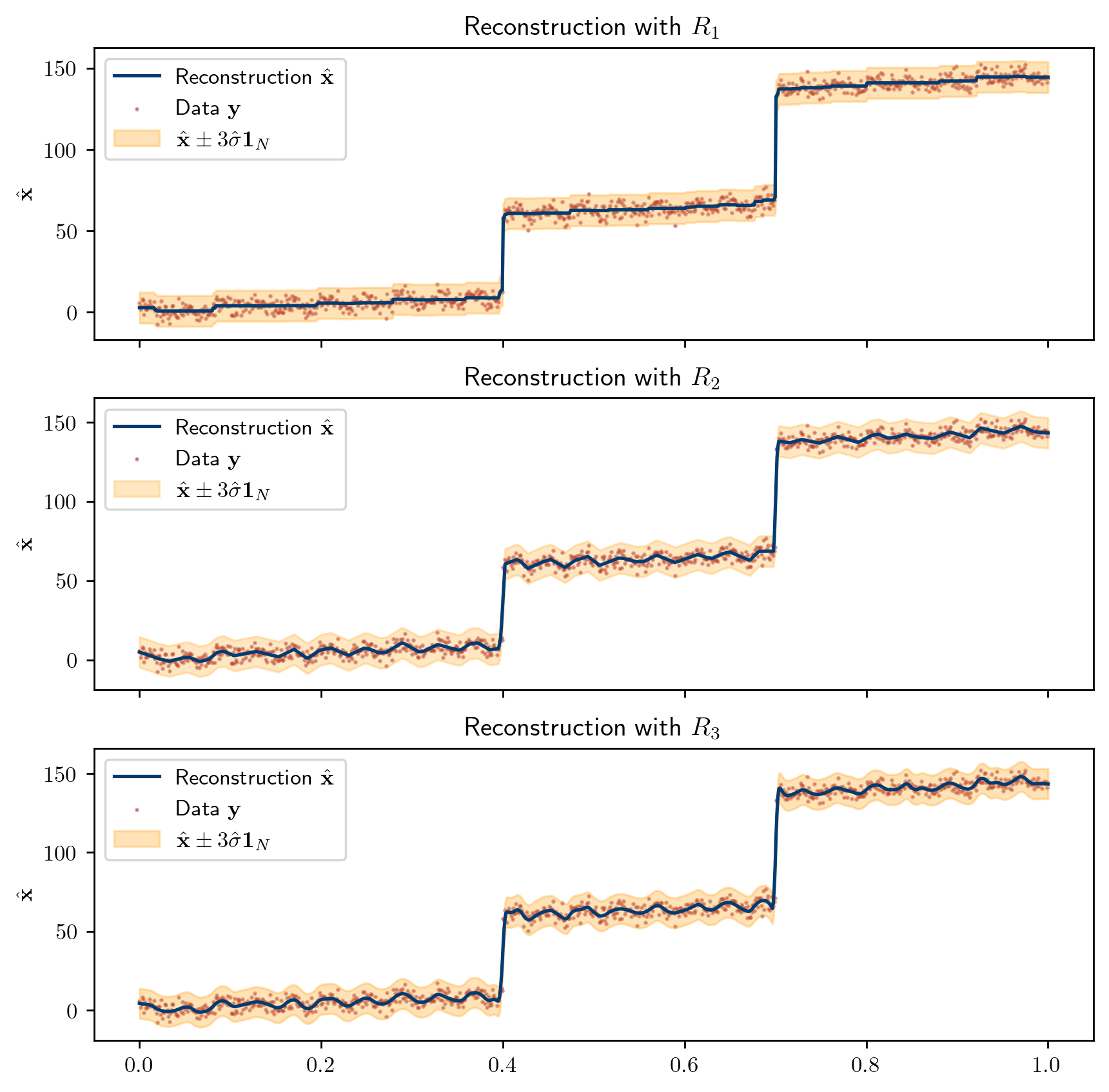}
    \caption{Results of applying \cref{alg:IAS_noise} to the signal denoising problem, according to three different prior sparsifying transformations. For comparison, we overlay the observed data as well as  $\pm 3$ standard deviation ($\hat{\sigma}$) intervals implied by the learned noise variance parameter $\hat{\nu} = \hat{\sigma}^2$. }
    \label{fig:test_1_reconstructions}
\end{figure}

In light of the reconstruction performance metrics in \cref{table:experiment_one_table_R1,table:experiment_one_table_R2,table:experiment_one_table_R3}, we select $\vartheta = 5 \cdot 10^{-1}$ as an ``optimal'' (according to the SSIM) parameter for all three sparsifying transformations and take a closer look at the reconstructions using this parameter. 
We emphasize that an automated regularization parameter selection for more precisely determining an optimal $\vartheta$ remains an open problem that will be investigated in future work. 

\Cref{fig:test_1_reconstructions} presents the reconstructions resulting from \cref{alg:IAS_noise} according to the three different sparsifying transformations given in \cref{eq:experiment_1_Rs}. 
Here, promoting sparsity under $R_1$, $R_2$, and $R_3$ corresponds to encouraging piecewise constant, linear, and quadratic behavior, respectively. We observe that the corresponding learned noise variances according to each sparsifying transformation are all within 4.1\% of the true noise variance $\overline{\nu} = 10$. 

Inspecting the implied $\pm 3$ standard deviation bars in \cref{fig:test_1_reconstructions}, it is evident that the learned noise variance parameters capture the deviation of the observed data $\mathbf{y}$ relative to the source estimate $\hat{\mathbf{x}}$. This suggests that our use of an uninformative noise variance hyper-prior leads to noise variance estimates that are consistent with the source estimate $\hat{\mathbf{x}}$ 
\emph{conditional on the proper choice of regularization parameter $\vartheta$.}

\begin{figure}[h!]
    \centering
    \includegraphics[width=.99\textwidth]{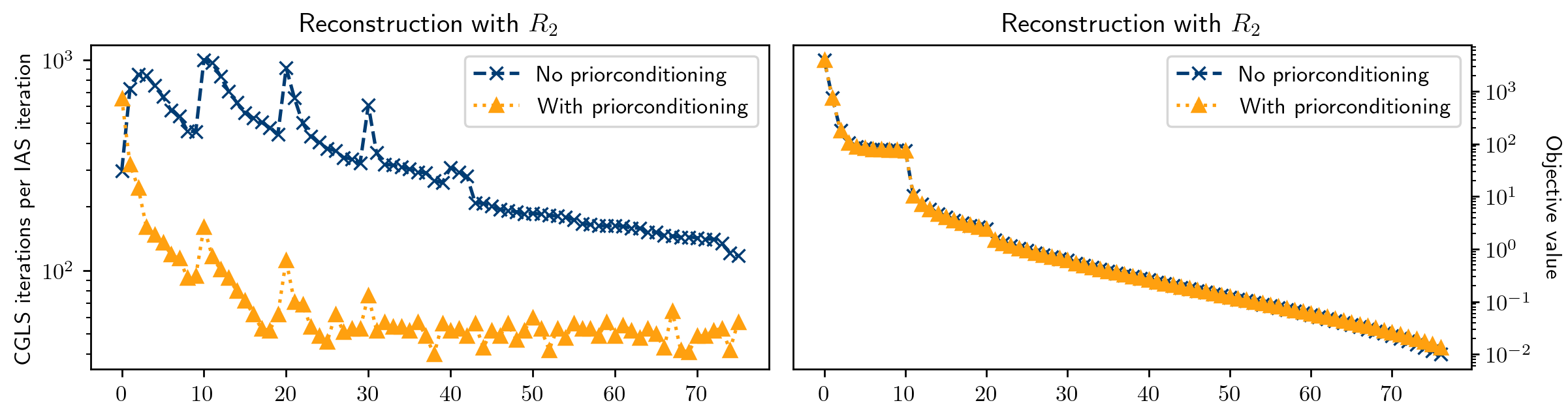}
    \caption{A comparison of the generalized IAS method with and without priorconditioning using the sparsifying transformation $R_2$ (the results using $R_1$ and $R_3$ are qualitatively similar). (Left) The number of CGLS iterations required at each outer IAS iteration. (Right) The objective value attained at each outer IAS iteration. The objective values have been offset by unimportant constants for visualization purposes.}
    \label{fig:test_1_cgls_objvals}
\end{figure}

Finally, \cref{fig:test_1_cgls_objvals} reports on the number of CGLS or PCGLS iterations required to compute each iteration of the generalized IAS algorithm, along with the objective value $\mathcal{G}(\mathbf{x}, \boldsymbol{\theta}, \nu)$ attained at each iteration. We find that the objective values of the algorithm with and without priorconditioning are nearly identical at each iteration. We further observe that priorconditioning reduces the number of CGLS iterations needed to satisfy the CGLS convergence criterion associated with the $\mathbf{x}$-updates in each case. 
The gradual decrease in the number of CGLS iterations is in part due to the warm start for each $\mathbf{x}$-update with the solution from the previous iteration. A second reason is that the priorconditioned least squares problems become better-conditioned as the estimate for $\boldsymbol{\theta}$ is refined.

\subsection{Computed tomography with an unknown noise variance}
\label{sub:ct_experiment}

We consider a CT inverse problem and compare the results of \cref{alg:IAS_noise} for a learned and fixed noise variance. 
We further compare the results for different hyper-prior parameters. 
For our numerical tests we use the $200 \times 200$ Shepp-Logan phantom image shown in \cref{fig:tomo_data}, which for computational purposes we view as a vectorized image $\overline{\mathbf{x}} \in \mathbb{R}^{40000}$. 
To set up the CT problem, we make use of the TRIPs-Py library \cite{pasha2024trips}. 
We define $F \in \mathbb{R}^{M \times N}$ to be a discretized Radon transformation corresponding to a parallel beam geometry, where $N = 40000$ and  $M = PQ$, with $P = 282$  the number of detector pixels, and $Q = 50$ the number of equispaced view angles in $[0, \pi]$ (oriented from the positive $y$-axis). 
The underdetermined rate for the problem given such $F$ is $M/N \approx 35\%$. For more details about CT and the specific problem formulation, we refer the reader to \cite{pasha2024trips, hansen2021computed}. 

To generate synthetic data for this experiment (and to avoid the inverse crime), we first define a second projection operator $\tilde{F} \in \mathbb{R}^{M \times \tilde{N}}$  using the same problem formulation but with a finer $600 \times 600$ grid ($\tilde{N} = 360000$) than is used to perform the reconstruction. We then generate a synthetic observation via $\mathbf{y} = \tilde{F} \overline{\mathbf{x}} + \mathcal{N}(\mathbf{0}_{M}, \overline{\nu} \mathbf{I}_{M})$. Here we have set the true noise variance equal to $3\%$ of the maximum of the noiseless transformed signal $F \overline{\mathbf{x}}$. The observation  (displayed as a sinogram) as well as baseline filtered backprojection and Tikhonov reconstructions are shown in \cref{fig:tomo_data}. 
Let
\begin{align}
\label{eq:R_one_dim}
    R_L 
    = \begin{bmatrix} R_1 \\ \mathbf{0}^T \end{bmatrix}
    \in \mathbb{R}^{L \times L}
\end{align}
be the one-dimensional discrete gradient matrix with reflexive boundary conditions. 
A two-dimensional, anisotropic discrete gradient operator with Neumann boundary conditions can then be expressed as
\begin{align}
\label{eq:RD2}
    R = \begin{bmatrix}
         R_{N_1} \otimes I_{N_2} \\ I_{N_1} \otimes R_{N_2}
    \end{bmatrix},
\end{align}
where $\otimes$ denotes the Kronecker product and $N_1 = N_2 = 200$. 
Observe that $\ker(R) = \operatorname{span}\{ \mathbf{1}_N\}$, meaning that $R$ possesses a nontrivial kernel of dimension one. 
However, $F$ and $R$ satisfy the common kernel condition in \cref{eq:common_kernel_condition} (this is checked numerically).

\begin{table}[h!]
\centering
\begin{tabular}{r r r r r r r}
\toprule
    & $r$ & $\beta$ & $\vartheta$ & $\tilde{r}$ & $\tilde{\beta}$ & $\vartheta$ \\
    \midrule
    $1$st prior model & $1$  & $1.5 + 10^{-3}$ & $10^{-1}$  & $-1$ & $1$ & $10^{-4}$  \\
    $2$nd prior model & $-1$  & $1$  & $5 \cdot 10^{-5}$  & $-1$  & $1$  &  $10^{-4}$ \\ \bottomrule
\end{tabular}
\caption{Hyper-prior parameters for the CT problem.}
\label{table:hyperprior_parameters_exp_two}
\end{table}

\cref{table:hyperprior_parameters_exp_two} displays the two sets of hyper-prior parameters that we use to perform the reconstruction. 
The first set of parameters imposes a gamma hyper-prior and loosely corresponds to $\ell_1$ regularization \cite{calvetti2019hierachical}. 
This first model would otherwise yield a convex problem if not for the noise variance hyper-prior parameters corresponding to an uninformative noise variance prior. 
The second set of parameters is chosen to impose an inverse gamma hyper-prior, resulting in a nonconvex problem even if the noise variance was held fixed. 
For both sets of hyper-prior parameters, the parameter $\vartheta$ has been hand-tuned.

It is computationally burdensome to implement priorconditioning for this experiment in the same way as in \Cref{sub:signal_denoising}. 
There, we implemented the required matrix-vector products for  $R_{\boldsymbol{\theta}}^\dagger$ approximated by \cref{eq:approx_pseudoinverse} using a Cholesky factorization followed by forward/back-substitution to apply the inverse of $R_{\boldsymbol{\theta}}^T R_{\boldsymbol{\theta}} + \delta I_N$. 
Although  $R_{\boldsymbol{\theta}}^T R_{\boldsymbol{\theta}} + \delta I_N$ is still sparse in this experiment, its bandwidth $p$ is related linearly to $N$ by $p = \sqrt{N}$. Hence performing a Cholesky factorization (even exploiting the sparsity) costs $\mathcal{O}(N^{3/2})$ flops. To overcome this obstacle, we resort to using the CG method for applying $R_{\boldsymbol{\theta}}^\dagger$ described earlier in \Cref{sec:priorconditioning}. That is, we use a CG method equipped with a fast DCT-based spectral preconditioner based on the (unweighted) operator $R^T R$. 
A detailed explanation of this technique is provided in \cref{app:iterative_pinv_computations}.

\begin{figure}[h!]
    \centering
    \includegraphics[width=.99\textwidth]{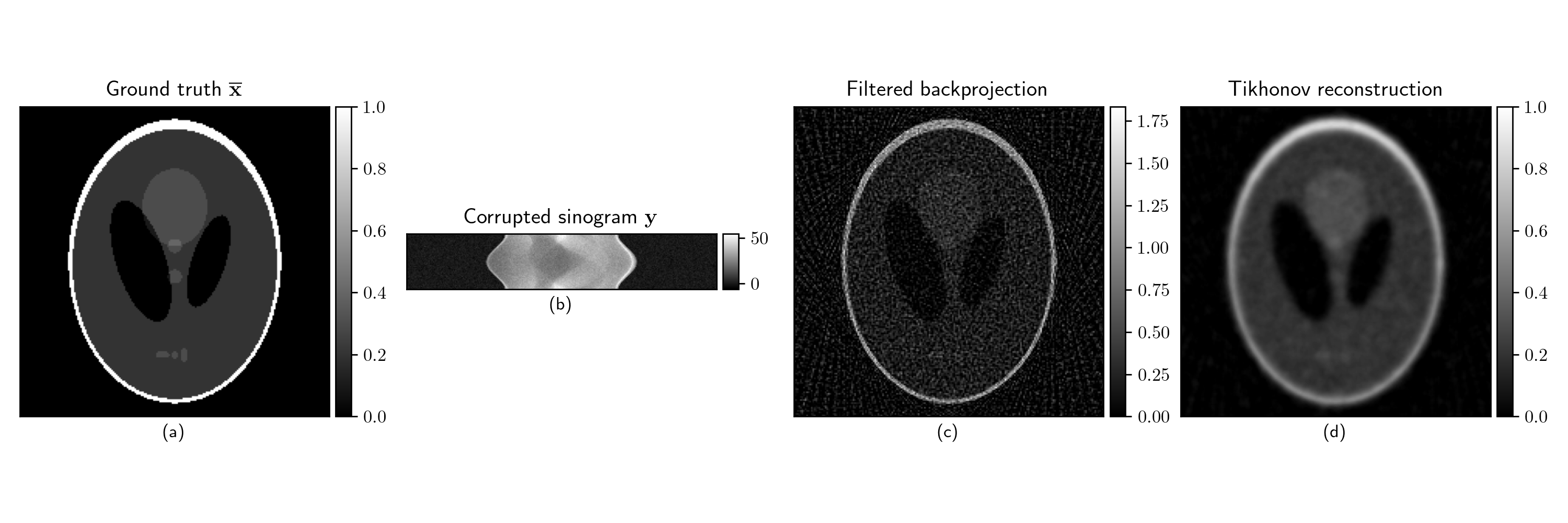}
    \caption{Setup for the tomography problem. (a) Ground truth phantom $\overline{\mathbf{x}}$. (b) Corrupted sinogram data. (c) Baseline filtered back-projection reconstruction. (d) Baseline Tikhonov reconstruction obtained via \cref{eq:tikhonov_initialization}. }
    \label{fig:tomo_data}
\end{figure}

\begin{figure}[h!]
    \centering
    \includegraphics[width=.95\textwidth]{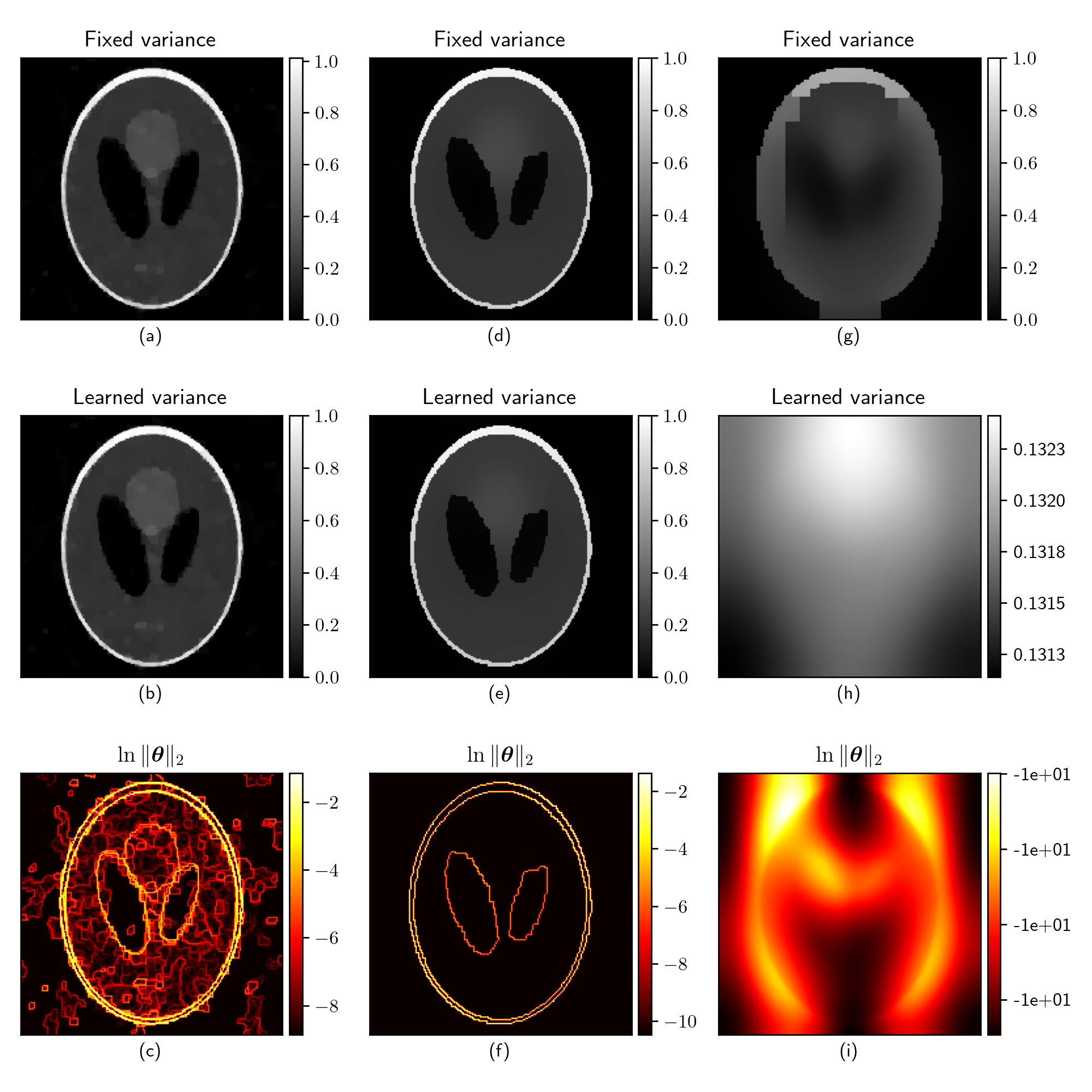}
    \caption{Comparison of reconstructions using fixed and learned noise variances. 
    First column: First hyper-prior parameter set and initialized using the Tikhonov solution in \cref{fig:tomo_data}. 
    Second column: Second hyper-prior parameter set and initialized using the corresponding solution in the first column. 
    Third column: Second hyper-prior parameter set and initialized with $\mathbf{x} = \mathbf{1}_N$. 
    Furthermore, the bottom row shows the natural logarithm of the pixel-wise norm of $\boldsymbol{\theta}$ corresponding to the solutions in the second row.}
    \label{fig:tomo_results}
\end{figure}

Various reconstruction results for the CT problem are shown in \cref{fig:tomo_results}. Here we make comparisons along three dimensions (1) the solutions for varying hyper-prior parameters; (2) the solutions for fixed versus learned noise variance; and (3) the solutions using the second hyper-prior parameter with both a ``good'' and ``bad'' initialization. To enforce nonnegativity in the reconstructions, each update to $\mathbf{x}$ is immediately followed by a projection onto the nonnegative orthant (e.g., see \cite[\S 5]{calvetti2020sparse} for a justification).

Comparing the hyper-priors, the reconstruction (b) using the first hyper-prior parameter set appears to better capture the features of the true phantom ($\text{SSIM} = $ $0.889$) than the reconstruction (e) using the second hyper-prior parameter set ($\text{SSIM} =$ $0.842$). Although the second reconstruction (e) misses some features, (f) indicates that this reconstruction is sparser (seen in the contrast in the components of $\boldsymbol{\theta}$) under $R$ compared to (c). 

Regarding fixed versus learned noise variances, \cref{alg:IAS_noise} with the first parameter set recovers a solution (b) with noise variance parameter $\hat{\nu} \approx 2.51$, which can be compared to the true signal variance  $\overline{\nu} = 2.60$. Although the learned variance is smaller than the truth, there is little qualitative difference between (b) and the solution obtained with fixed noise variance in (a). Note that since in (a) the noise variance is fixed, by \cref{thm:main_convexity} this solution corresponds to the unique global minimizer of a convex objective and is the only provably convex problem solved in this experiment. A similar observation can be made about the reconstruction (e) using the second parameter set and a good initialization; here the learned noise variance is $\hat{\nu} \approx 3.06$ which is greater than the truth, yet there is little qualitative difference when compared with (d).

For the second hyper-prior parameter set, which promotes sparsity more strongly, we illustrate the impact of the choice of initialization by comparing the results (second column)  obtained by initializing with the solution from the first parameter set (first column), versus the results (third column) obtained using a bad initialization chosen as the constant vector $\mathbf{x} = \mathbf{1}_N$. We find that the solution with learned noise variance (h) finds a local minimum with a very large noise variance $(\hat{\nu} \approx 15247$) and a source $\hat{\mathbf{x}}$ that is nearly constant. A similar default to a constant signal has been observed in \cite{Dong2023Horshoe, sanders2020effective} and is typical of a poor initialization in the nonconvex setting considered here. Although the analogous solution obtained with fixed noise variance (g) does not default to a constant, here the solver finds a local minimum with greater objective value than the solution (d) obtained using the same hyper-prior parameters but using (a) as the solver initialization. Overall, this comparison underscores the importance of picking a good initialization when utilizing a strongly sparsity-promoting hyper-prior. 
\section{Summary}\label{sec:summary}

We generalized traditional sparsity-promoting hierarchical Bayesian models and their MAP estimation using the IAS algorithm in two ways:
(1) We expanded the IAS framework to more general sparsifying transforms, which do not necessarily have  trivial kernels; 
(2) We allowed treating the noise variance as a random variable that is estimated during the inference procedure. The resulting generalized IAS algorithm arises from straightforward modifications to the original algorithm.
We demonstrated that these augmentations did not significantly burden the computational expense of the algorithm, and moreover  bring only small modifications to the convexity and convergence analyses of the original IAS procedure. Although not demonstrated here, our approach preserves the option to amalgamate our generalized IAS algorithms with hybrid strategies for nonconvex models, as elucidated in \cite{calvetti2020sparsity,si2024path}. To reduce the computational cost of the method, we discussed a generalization of the priorconditioning technique and detailed its efficient implementation for large-scale problems. Future work will include designing sampling strategies to explore the complete posterior distribution, a critical direction recently initiated in \cite{calvetti2024computationally}. Such sampling strategies may be accelerated using the priorconditioning technique and its efficient implementation detailed here. Additionally, automated methods for determining an appropriate hyper-prior (regularization) parameter $\vartheta$ will be investigated.

\section*{Acknowledgments}
JL and AG were supported in part by AFOSR grant \#F9550-22-1-0411 and US DOD (ONR MURI) grant \#N00014-20-1-2595. 
JG was supported in part by the US DOD (ONR MURI) grant \#N00014-20-1-2595. AG was also supported in part by NSF grants DMS \#1912685 and DOE ASCR \#DE-ACO5-000R22725.

\section*{Data availability statement}
The data that support the findings of this study are openly available at the following URL: \url{https://github.com/jlindbloom/GeneralizedSparsitySolvers}.

\appendix
\section{Proofs of \cref{thm:IAS_G_convergence_convex_case} and \cref{thm:IAS_G_limit_points}}\label{app:convergence_part_one}

Here we prove \cref{thm:IAS_G_convergence_convex_case} and \cref{thm:IAS_G_limit_points}. To prove the latter, we first establish two lemmas that apply in the nonconvex case. The idea is that since $\operatorname{dom}(\mathcal{G})$ is not closed, it is not straightforward to apply existing results to the convergence of block coordinate descent applied to the minimization of $\mathcal{G}$. However, it is sufficient to consider instead the minimization of an auxiliary function $\mathcal{H}$ for which $\operatorname{dom}(\mathcal{H})$ is closed. 

\begin{lemma}\label{lem:IAS_G_epsilon} 
    Let $\mathcal{G}$ denote the objective in \cref{eq:IAS_G}. 
    Then $(\mathbf{x}^*, \boldsymbol{\theta}^*)$ is a local minimizer of $\mathcal{G}$ if and only if it is a coordinate-wise minimizer of $\mathcal{G}$. The latter means that $\boldsymbol{\theta}^{*} = \argmin_{\boldsymbol{\theta}} \mathcal{G}(\mathbf{x}^*, \boldsymbol{\theta})$ and $\mathbf{x}^{*} = \argmin_{\mathbf{x}} \mathcal{G}(\mathbf{x}, \boldsymbol{\theta}^{*})$. Moreover, if $(\mathbf{x}^*, \boldsymbol{\theta}^*)$ is a local minimizer of $\mathcal{G}$ then $(\mathbf{x}^*, \boldsymbol{\theta}^*) \in \mathbb{R}^{N} \times [\varepsilon, +\infty)^{N}$, where $\varepsilon = \vartheta_{\textrm{min}} (\eta/r)^{1/r}$.
\end{lemma}
\begin{proof} 

Observe by inspection of \cref{eq:IAS_G} that $\mathcal{G}$ is differentiable on $\operatorname{Int}(\operatorname{dom}(\mathcal{G})) = \mathbb{R}^N \times \mathbb{R}_{++}^N$. Furthermore, any stationary point of $\mathcal{G}$ must lie in $\operatorname{Int}(\operatorname{dom}(\mathcal{G}))$. 
Thus, $(\mathbf{x}^*, \boldsymbol{\theta}^*)$ is a local minimizer of $\mathcal{G}$ if and only if $\nabla_{\mathbf{x}} \mathcal{G}(\mathbf{x}^*, \boldsymbol{\theta}^*) = \mathbf{0}_N$ and  $\nabla_{\boldsymbol{\theta}} \mathcal{G}(\mathbf{x}^*, \boldsymbol{\theta}^*) = \mathbf{0}_N$. 
Suppose first that $(\mathbf{x}^*, \boldsymbol{\theta}^*)$ is a local minimizer of $\mathcal{G}$. Since the objective function appearing in the $\mathbf{x}$-update in \cref{eq:IAS_x_update} is smooth and strictly convex, $\mathbf{x}^*$ must be the unique solution to $\argmin_{\mathbf{x}} \mathcal{G}(\mathbf{x}, \boldsymbol{\theta}^*)$. 
Although the objective function appearing in the $\boldsymbol{\theta}$-update is not always strictly convex, from the discussion following \cref{eq:IAS_update_beta_general}, it has a unique minimizer for any $\mathbf{x}$. 
This implies that $\boldsymbol{\theta}^*$ must be the unique solution to $\argmin_{\boldsymbol{\theta}} \mathcal{G}(\mathbf{x}^*, \boldsymbol{\theta})$. 
Thus, if $(\mathbf{x}^*, \boldsymbol{\theta}^*)$ is a local minimizer of $\mathcal{G}$ then it must be a coordinate-wise minimizer of $\mathcal{G}$. The reverse direction follows similarly and is due to the smoothness of $\mathcal{G}$ (see Lemma 3.3 of \cite{calvetti2020sparse} for a similar statement). 
Finally,  from the lower bound in \cref{eq:theta_lower}  we have $(\mathbf{x}^*, \boldsymbol{\theta}^*) \in \mathbb{R}^{N} \times [\varepsilon, +\infty)^{N}$.
\end{proof}

\begin{lemma}\label{lem:IAS_H_properties} Let $\mathcal{G}$ denote the objective in \cref{eq:IAS_G}, and let $\mathcal{H}(\mathbf{x}, \boldsymbol{\theta}):\mathbb{R}^{2N} \to \mathbb{R} \cup \{ +\infty\}$ be the extended real-valued function given by $\mathcal{H}(\mathbf{x}, \boldsymbol{\theta}) = \mathcal{G}(\mathbf{x}, \boldsymbol{\theta}) + \delta_{[\varepsilon, +\infty)^N}(\boldsymbol{\theta})$ with $\varepsilon = \vartheta_{\text{min}} (\eta/r)^{1/r}$. Then $\mathcal{H}$ is proper, closed, coercive, and has bounded level sets. Moreover, $(\mathbf{x}^*, \boldsymbol{\theta}^*)$ is a local minimizer of $\mathcal{H}$ if and only if it is a coordinate-wise minimizer of $\mathcal{H}$.  Finally, the functions $\mathcal{G}$ and $\mathcal{H}$ attain the same minimal value on $\mathbb{R}^{2N}$ at the same stationary point, and their sets of stationary points are the same.
\end{lemma}
\begin{proof}
    First, by definition, $\mathcal{H}$ is proper since a value of $-\infty$ is never attained and $\operatorname{dom}(\mathcal{G})$ is nonempty. Second, note that $\mathcal{H}$ is continuous on $\operatorname{dom}(\mathcal{H}) = \mathbb{R}^{N} \times [\varepsilon, +\infty)^{N}$ and that $\operatorname{dom}(\mathcal{H})$ is closed. 
    Hence, $\mathcal{H}$ is closed (see, e.g., Theorem 2.8 of \cite{beck2017first}). 
    Third, by observing \cref{eq:IAS_G}, we see that $\mathcal{G} \to +\infty$ as $\| (\mathbf{x}, \boldsymbol{\theta}) \|_2 \to +\infty$. 
    This implies that $\mathcal{H}$ is coercive. 
    Finally, coercivity yields that $\mathcal{H}$ has bounded level sets, i.e., $\operatorname{Lev}(\mathcal{H}, \alpha) = \{ (\mathbf{x}, \boldsymbol{\theta}) \in \mathbb{R}^{2N} \mid \mathcal{H}(\mathbf{x}, \boldsymbol{\theta}) \leq \alpha \}$ is bounded for every $\alpha \in \mathbb{R}$. Since $\mathcal{H}$ is proper, closed, and coercive and $\mathbb{R}^{2N}$ is closed, it follows from the Weierstrass theorem (see \cite[Theorem 2.14]{beck2017first}) that $\mathcal{H}$ attains its minimal value over $\mathbb{R}^{2N}$. 
    
    For the next part of the lemma, let $(\mathbf{x}^*, \boldsymbol{\theta}^*)$ be a local minimizer of $\mathcal{H}$. 
    By definition, this implies that $\mathbf{0}_{2N} \in \partial \mathcal{H}(\mathbf{x}^*, \boldsymbol{\theta}^*)$. This, in turn, requires either $\nabla_{\mathbf{x}, \boldsymbol{\theta}} \mathcal{H}(\mathbf{x}^*, \boldsymbol{\theta}^*) = \mathbf{0}_{2N}$ if $(\mathbf{x}^*, \boldsymbol{\theta}^*) \in \operatorname{Int}(\operatorname{dom}(\mathcal{H}))$, or $(\mathbf{x}^*, \boldsymbol{\theta}^*) \in \operatorname{Bd}(\operatorname{dom}(\mathcal{H}))$ otherwise (this holds if $\theta_i^* = \varepsilon$ for at least one of the $\theta_i^*$). 
    In both cases, invoking the uniqueness of the minimizer for each coordinate update and applying logic similar to that in the proof of \cref{lem:IAS_G_epsilon}, we have that $(\mathbf{x}^*, \boldsymbol{\theta}^*)$ is also a coordinate-wise minimizer of $\mathcal{H}$. The other direction follows similarly.
    
    The last observation to make is that by \cref{lem:IAS_G_epsilon}, we then have that the sets of stationary points of $\mathcal{H}$ and $\mathcal{G}$ are the same. Thus, $\mathcal{G}$ also attains its minimal value on $\mathbb{R}^{2N}$.
\end{proof}

Notably, \cref{lem:IAS_G_epsilon} and \cref{lem:IAS_H_properties} guarantee that the limit points of the iterates produced by \cref{alg:IAS} must be stationary points, even in the nonconvex case.
We can now prove \cref{thm:IAS_G_limit_points}, which for convenience is restated below. 
This will follow with a proof of \cref{thm:IAS_G_convergence_convex_case}.

\begin{proof}[Proof of \cref{thm:IAS_G_limit_points}]
    Although it is possible that $(\mathbf{x}^{(0)}, \boldsymbol{\theta}^{(0)}) \not \in \mathbb{R}^{N} \times [\varepsilon, +\infty)^N$, the subsequence $\{ (\mathbf{x}^{(k)}, \boldsymbol{\theta}^{(k)})  \}_{k \geq 1}$ must be contained in $\mathbb{R}^{N} \times [\varepsilon, +\infty)^N$ since every future coordinate update $\boldsymbol{\theta}^{(k+1)}$ must be contained in $[\varepsilon, +\infty)^N$, where $\varepsilon = \vartheta_{\text{min}} (\eta/r)^{1/r}$. With $\mathcal{H}(\mathbf{x}, \boldsymbol{\theta}) = \mathcal{G}(\mathbf{x}, \boldsymbol{\theta}) + \delta_{[\varepsilon, +\infty)^N}$, this implies that $\mathcal{H}(\mathbf{x}^{(k)}, \boldsymbol{\theta}^{(k)}) = \mathcal{G}(\mathbf{x}^{(k)}, \boldsymbol{\theta}^{(k)})$ for $k \geq 1$. 
    From our discussion in \cref{sub:background_MAP}, note that the solutions to both the coordinate updates for $\mathbf{x}$ and $\boldsymbol{\theta}$ are always unique. 
    Moreover, from \cref{lem:IAS_H_properties}, we have that $\mathcal{H}$ is a proper closed function that is continuous on $\operatorname{dom}(\mathcal{H}) = \mathbb{R}^{N} \times [\varepsilon, +\infty)^N$ and has bounded level sets. 
    These properties imply that $\{ (\mathbf{x}^{(k)}, \boldsymbol{\theta}^{(k)} )\}$ is bounded, and furthermore that any limit point of this sequence is a coordinate-wise minimizer of $\mathcal{H}$ (and also of $\mathcal{G}$). 
    By \cref{lem:IAS_G_epsilon}, any such limit point must also be a local minimizer of $\mathcal{G}$. 
\end{proof}

The fact that \cref{alg:IAS} converges to the unique global minimizer $(\mathbf{x}^{\MAP}, \boldsymbol{\theta}^{\MAP})$ of $\mathcal{G}$ then follows from the uniqueness of the stationary point.


\begin{proof}[Proof of \cref{thm:IAS_G_convergence_convex_case}]
    Since $r \geq 1$ and $\eta > 0$, \cref{thm:previous_convexity_result} implies that $\mathcal{G}$ is globally strictly convex.
    Furthermore, from \cref{thm:IAS_G_limit_points}, we have that $\{ (\mathbf{x}^{(k)}, \boldsymbol{\theta}^{(k)}) \}$ is bounded, and that any of its limit points are stationary points of $\mathcal{G}$. Moreover, since $\mathcal{G}$ is strictly convex and coercive, there exists exactly one such stationary point $(\mathbf{x}^{\MAP}, \boldsymbol{\theta}^{\MAP})$ at which $\mathcal{G}$ achieves its global minimum value. 
    Note that $\{ (\mathbf{x}^{(k)}, \boldsymbol{\theta}^{(k)} ) \}$ satisfies a descent property. 
    That is, the sequence of objective values $\{ \mathcal{G}(\mathbf{x}^{(k)}, \boldsymbol{\theta}^{(k)} ) \}$ is nonincreasing. 
    It is straightforward to show that the strict convexity of $\mathcal{G}$ implies the stronger descent property
    \begin{align*}
        \mathcal{G}(\mathbf{x}^{(k)}, \boldsymbol{\theta}^{(k)}) > \mathcal{G}(\mathbf{x}^{(k)}, \boldsymbol{\theta}^{(k+1)}) > \mathcal{G}(\mathbf{x}^{(k+1)}, \boldsymbol{\theta}^{(k+1)})
    \end{align*}
     if $\mathcal{G}(\mathbf{x}^{(k)}, \boldsymbol{\theta}^{(k)} ) \neq \mathcal{G}(\mathbf{x}^{\MAP}, \boldsymbol{\theta}^{\MAP} )$, i.e., that the sequence of objective values is strictly decreasing unless the minimum value is reached. This implies that $\{ \mathcal{G}(\mathbf{x}^{(k)}, \boldsymbol{\theta}^{(k)} ) \} \to  \mathcal{G}(\mathbf{x}^{\MAP}, \boldsymbol{\theta}^{\MAP} )$ as $k \to \infty$. 
 Observe that since $\mathcal{G}$ is strictly convex, for any $\delta > 0$ there exists some $\varepsilon > 0$ such that 
 \begin{align}
    \| ( \mathbf{x}^{(k)}, \boldsymbol{\theta}^{(k)} ) - ( \mathbf{x}^{\MAP}, \boldsymbol{\theta}^{\MAP} )  \| > \delta \implies | \mathcal{G}(\mathbf{x}^{(k)}, \boldsymbol{\theta}^{(k)}) - \mathcal{G}(\mathbf{x}^{\MAP}, \boldsymbol{\theta}^{\MAP} ) | > \varepsilon.
\end{align}
 Suppose that $\{ (\mathbf{x}^{(k)}, \boldsymbol{\theta}^{(k)}) \} \nrightarrow (\mathbf{x}^{\MAP}, \boldsymbol{\theta}^{\MAP})$ as $k \to \infty.$ Then there exists some $\delta > 0$ such that $\| ( \mathbf{x}^{(k)}, \boldsymbol{\theta}^{(k)} ) - ( \mathbf{x}^{\MAP}, \boldsymbol{\theta}^{\MAP} )  \| > \delta$ for infinitely many $k$.  However, by the initial observation, this implies the existence of an $\varepsilon > 0$ such that 
    \begin{align} 
    | \mathcal{G}(\mathbf{x}^{(k)}, \boldsymbol{\theta}^{(k)} ) - \mathcal{G}(\mathbf{x}^{\MAP}, \boldsymbol{\theta}^{\MAP}) | > \varepsilon 
    \end{align}
    for infinitely many $k$, contradicting the fact that $\{ \mathcal{G}(\mathbf{x}^{(k)}, \boldsymbol{\theta}^{(k)} ) \} \to  \mathcal{G}(\mathbf{x}^{\MAP}, \boldsymbol{\theta}^{\MAP})$ as $k \to \infty$.
    Thus it must be the case that $\{ (\mathbf{x}^{(k)}, \boldsymbol{\theta}^{(k)}) \} \to (\mathbf{x}^{\MAP}, \boldsymbol{\theta}^{\MAP})$ as $k \to \infty$.
    \end{proof}

 \section{Iterative computation of pseudoinverses}\label{app:iterative_pinv_computations}

We discuss iterative methods for computing matrix-vector products with the pseudoinverses $R_{\boldsymbol{\theta}}^{\dagger}$ and $(R_{\boldsymbol{\theta}}^\dagger)^T$. 
For ease of notation, we momentarily drop the dependence of $R_{\boldsymbol{\theta}}$ on $\boldsymbol{\theta}$, simply denoting it by $R$. 
The matrix-vector product $R^\dagger \mathbf{v}$ can be computed by applying the CG method to the solution of
\begin{align}\label{eq:RtR}
    R^T R \mathbf{z} = R^T \mathbf{v}
\end{align} 
for $\mathbf{z} \in \operatorname{col}(R^T)$, intialized at some $\mathbf{z}^{(0)} \in \operatorname{col}(R^T)$ \cite{hestenes1975pseudoinversus, hayami2018convergence}. To this end, the initial value can be simply chosen as $\mathbf{z}^{(0)} = \mathbf{0}_N$. In exact arithmetic, this procedure terminates at some iteration index $k^\star \leq N$ such that $\mathbf{z}^{(k^\star)} = R^\dagger \vecv$. Initializations with $\mathbf{z}^{(0)} \not \in \operatorname{col}(R^T)$ generally cause this procedure to fail, however. 

The matrix-vector product $(R^\dagger)^T \mathbf{v}$ can be computed in a similar way. Recall that the pseudoinverse commutes with transposition, i.e., that $(R^\dagger)^T = (R^T)^\dagger$. Using an analogous construction to \cref{eq:RtR}, $(R^\dagger)^T \mathbf{v}$ can be computed by applying the CG method to the solution of
\begin{align}\label{eq:RRt}
    R R^T \mathbf{z} = R \mathbf{v}
\end{align}
for $\mathbf{z} \in \operatorname{col}(R)$. This time, the initial value should be selected so that $\mathbf{z}^{(0)}  \in  \operatorname{col}(R)$. In exact arithmetic, this procedure again terminates at some iteration index $k^\star \leq K$ such that $\mathbf{z}^{(k^\star)} = (R^\dagger)^T \mathbf{v}$.

A preconditioner can be used to accelerate the convergence of both CG methods. To this end, we note that while nonsingular symmetric preconditioners for singular symmetric systems have been studied (e.g., see \cite{Kaasschieter1988Preconditioned, Pearson2020Preconditioners}),  preconditioning a singular symmetric system with a {\em singular} symmetric preconditioner has been studied considerably less (e.g., see \cite{Pearson2020Preconditioners, Elden2012GMRESSingular, Ranjbar2014CauchyVariableCoeff}). Thus, because of its apparent usefulness in accelerating the convergence of \cref{alg:IAS_noise}, we introduce such a preconditioner here.  First, we provide the following lemma.

\begin{lemma}\label{lem:nonsingular_preconditioning}
    Let $A, M \in \mathbb{R}^{N \times N}$ be symmetric matrices such that $A, M \succeq 0$ (the matrices are positive semi-definite) and $\operatorname{col}(M) = \operatorname{col}(A)$, and let $\mathbf{b} \in \operatorname{col}(A)$. Let $L$ be any matrix such that $M = L L^T$, such as (but not limited to) the spectral square root. Then the problem
    \begin{align}\label{eq:pinv_prob1}
      \text{find } \mathbf{x} \in \operatorname{col}(A) \text{  s.t. }  A \mathbf{x} = \mathbf{b}
    \end{align}
    has a the unique solution $\mathbf{x}^\star = (L^{T})^\dagger \mathbf{z}^\star$, where $\mathbf{z}^\star$ is the unique solution to the problem
    \begin{align}\label{eq:pinv_prob2}
        \text{find }  \mathbf{z} \in \operatorname{col}(A) \text{  s.t. } L^{\dagger} A (L^T)^\dagger \mathbf{z} = L^\dagger \mathbf{b}.
    \end{align}

\end{lemma}

\begin{proof}
    We begin by recalling that if $C \in \mathbb{R}^{N \times N}$ is a symmetric matrix then its  pseudoinverse viewed as the restricted map 
    $$C^\dagger : \operatorname{col}(C) \to \ker(C)^\perp$$
    is a bijection. 
    Moreover, since $C$ is symmetric, by the fundamental subspace theorem, we have $\ker(C)^\perp = \operatorname{col}(C^T) = \operatorname{col}(C)$. Thus $C^\dagger$ is a bijection when viewed as a transformation from the range of $C$ to itself. In the context of \cref{eq:pinv_prob1}, since $A$ is symmetric and $\mathbf{b} \in \operatorname{col}(A)$, the unique solution is given by $\mathbf{x}^\star = A^\dagger \mathbf{b}$.
    
    Next, observe that $\operatorname{col}(L) = \operatorname{col}(M) = \operatorname{col}(A)$. Since $L^\dagger$ viewed as a function $L^\dagger : \operatorname{col}(A) \to \operatorname{col}(A)$ is a bijection, there must exist some $\mathbf{z}^\star \in \operatorname{col}(A)$ such that $\mathbf{z}^\star = L^T \mathbf{x}^\star$ and $\mathbf{x}^\star = (L^T)^\dagger \mathbf{z}^\star$. 
    Due to the bijection property of $A^\dagger$, which extends to $L^\dagger A (L^T)^\dagger$, \cref{eq:pinv_prob2} has the unique solution $\mathbf{z}^\star$ given by 
    \begin{align}
        L^\dagger A (L^T)^\dagger \mathbf{z}^\star =  L^\dagger A \mathbf{x}^\star = L^\dagger \mathbf{b}.
    \end{align}
\end{proof}
Similar to the preconditioned CG method for the nonsingular symmetric case, by examining the preconditioned CG method resulting from \cref{lem:nonsingular_preconditioning}, we observe that the method can be written such that it avoids referencing $L^\dagger$ altogether and only requires matrix-vector products with $M^{\dagger}$. Thus, the preconditioned CG method for the singular symmetric system is just the standard preconditioned CG method with the simple modifications of replacing $M^{-1}$ with $M^\dagger$ and requiring that the method is initialized with $\mathbf{x}^{(0)} \in \operatorname{col}(A)$.
Here, these CG methods for computing pseudoinverse matrix-vector products are embedded as an inner-loop within an outer-loop CG method computing the solution to \cref{eq:IAS_x_update_least_whitened}. Note that these CG methods for computing matrix-vector products with pseudoinverses must be equipped with a stopping criterion. Furthermore, because of the inherent nature of iterative methods, these matrix-vector products will not be computed exactly. As such, some care is needed w.r.t. the relationship between the stopping criteria of the inner- and outer-loop CG methods. For one such study, see \cite{Golub1999Inexact}.

\subsection*{Case study: Two-dimensional Neumann gradient with preconditioning}

Although our method is suitable for any sparse transform operator $R$ satisfying \cref{eq:common_kernel_condition}, we now focus on the case where $R$ represents a two-dimensional discrete gradient operator equipped with Neumann boundary conditions (given by \cref{eq:RD2}) for a uniform $N_1 \times N_2$ grid. 
The operator $R^T R$ can then be interpreted (up to a scale factor) as a discretized Laplacian operator $\Delta$. Similarly, $R_{\boldsymbol{\theta}}^T R_{\boldsymbol{\theta}}$ can be interpreted as a weighted Laplacian $\nabla \cdot ( \mathbf{w}(\mathbf{x}) \nabla )$. 

If the grid is sufficiently large, such as in the imaging context, the aforementioned CG methods for computing matrix-vector products with $R_{\boldsymbol{\theta}}^\dagger$ and $(R_{\boldsymbol{\theta}}^\dagger)^T$ may require many iterations until a sufficient approximation is reached. We would like to implement a preconditioner for the CG method to reduce the computational burden. A natural spectral preconditioner for the matrix $R_{\boldsymbol{\theta}}^T R_{\boldsymbol{\theta}}$ can be derived from the unweighted Laplacian $R^T R$. The matrix  $R^T R$ is a sum of specially-structured matrices,\footnote{Specifically, in the Neumann boundary condition case $R^T R$ can be written as the sum of block Toeplitz with Toeplitz blocks (BTTB), block Toeplitz with Hankel blocks (BTHB), block Hankel with Toeplitz blocks (BHTB), and block Hankel with Hankel blocks (BHHB) matrices \cite{hansen2006deblurring}. In the Dirichlet boundary condition case, $R^T R$ is a block Toeplitz with Toeplitz blocks (BTTB) matrix and can be diagonalized by the type I discrete sine transform (DST). In the periodic boundary condition case, $R^T R$ is a block circulant with circulant blocks (BCCB) matrix and can be diagonalized by the discrete Fourier transform (DFT).} such that it can be diagonalized \emph{a priori} by the (orthonormal, type II) two-dimensional discrete cosine transform (DCT) (e.g., see \cite{hansen2006deblurring, Strang1999DCT, Makhoul1980FastCosine}). Specifically, letting $B$ denote the DCT for a $N_1 \times N_2$ grid, it holds that
\begin{align}\label{eq:dct_diagonalization}
  M = R^T R = B^T \Lambda B,
\end{align}
where $\Lambda$ is a diagonal matrix with nonnegative entries containing the eigenvalues of $R^T R$, and $B^T = B^{-1}$ denotes the inverse DCT. Note that the eigenvalues are quickly computed as $\Lambda = \operatorname{diag}( (B R^T R B^T \mathbf{v}) \oslash \mathbf{v} )$ for a vector $\mathbf{v} \in \mathbb{R}^N$ with nonzero entries and $\oslash$ denoting component-wise division. We are therefore motivated to use $M^\dagger = (R^T R)^{\dagger} = B^T \Lambda^\dagger B$ as a preconditioner when applying $R_{\boldsymbol{\theta}}^\dagger$ via the CG method applied to \cref{eq:RtR} since it can be applied directly and cheaply in $\mathcal{O}(N_1 N_2 \log(N_1 N_2))$ flops.
Consider \cref{lem:nonsingular_preconditioning} with $A = R_{\boldsymbol{\theta}}^T R_{\boldsymbol{\theta}}$ and $\mathbf{b} = R_{\boldsymbol{\theta}}^T \mathbf{v}$. Clearly $\mathbf{b} \in \operatorname{col}(A)$. Now, let $L = B^T \Lambda^{1/2}$ in \cref{eq:dct_diagonalization}. \cref{lem:nonsingular_preconditioning} then justifies the use of the preconditioned CG method for a singular symmetric system to compute the matrix-vector product $R_{\boldsymbol{\theta}}^\dagger \mathbf{v}$ with preconditioner $M^\dagger = B^T \Lambda^{\dagger} B$. We also require a method for computing matrix-vector products of the form $(R_{\boldsymbol{\theta}}^\dagger)^T \mathbf{v}$. 
One option is to take the approach described via \cref{eq:RRt}, wherein the CG method is applied to the solution of
\begin{align}\label{eq:zsolve}
    R_{\boldsymbol{\theta}} R_{\boldsymbol{\theta}}^T \mathbf{z} = R_{\boldsymbol{\theta}} \mathbf{v}
\end{align}
for $\mathbf{z} \in \operatorname{col}(R)$, with an initialization $\mathbf{z}^{(0)} \in \operatorname{col}(R)$. Unfortunately, $R_{\boldsymbol{\theta}} R_{\boldsymbol{\theta}}^T$ is no longer a Laplacian operator (in the continuous setting, this corresponds to the gradient of the divergence), and unlike in the previous case, \emph{there is no exploitable structure} in the corresponding unweighted operator $R R^T$ from which we can derive a spectral preconditioner. 
However, by use of pseudoinverse identities it is in fact possible to ``recycle'' the same spectral preconditioner as used for $R_{\boldsymbol{\theta}}^T R_{\boldsymbol{\theta}}$. In this regard observe that for any $C \in \mathbb{R}^{K \times N}$, we have  
$$(C^\dagger)^T = (C^T)^\dagger \text{ and } (C^\dagger)^T = C (C^T C)^\dagger,$$
so that the solution $\mathbf{z}$ of \cref{eq:zsolve} can be expressed as  $\mathbf{z} = R_{\boldsymbol{\theta}} ( R_{\boldsymbol{\theta}}^T R_{\boldsymbol{\theta}} )^\dagger \mathbf{v}$. Decomposing the CG method as two algorithms acting separately on the column and kernel spaces  \cite{hayami2018convergence}, the matrix-vector product  $\mathbf{u} = (R_{\boldsymbol{\theta}}^T R_{\boldsymbol{\theta}})^\dagger \mathbf{v}$ can be computed by applying the CG method to the solution of the singular system
\begin{align}
    R_{\boldsymbol{\theta}}^T R_{\boldsymbol{\theta}} \mathbf{u} = (I_N - W W^\dagger) \mathbf{v}
\end{align}
for $\mathbf{u} \in \operatorname{col}(R^T)$, with any initialization $\mathbf{u}^{(0)} \in \operatorname{col}(R^T)$ and any matrix $W$ such that $\operatorname{col}(W) = \ker(R)$. Note that $W$ can be chosen to be the same matrix described earlier in \cref{sec:obliqueprojection}, and that $I_N - W W^\dagger$ represents an orthogonal projection operator onto the subspace $\ker(R)^\perp = \operatorname{col}(R^T)$. 

We conclude this section with a brief numerical experiment demonstrating the efficacy of the spectral preconditioner for the singular system. Letting $N_1 = N_2$ and $N = N_1^2$,  \cref{fig:lapinv_results} provides the number of CG iterations and wall-clock time needed to compute the matrix-vector product $R_{\boldsymbol{\theta}}^\dagger \mathbf{v}$ for varying $N$. Here we have averaged the results over 50 test vectors $\mathbf{v} \sim \mathcal{N}(\veczero_N, I_N)$ and hyperparameter vectors $\theta_i \overset{\text{ind}}{\sim} \mathcal{U}([1,50])$ for $i = 1, \ldots, K$, terminating the CG methods at the first iteration $k^\star$ such that 
\begin{align}
    \frac{ \| R_{\boldsymbol{\theta}}^T R_{\boldsymbol{\theta}} \mathbf{z}^{(k^\star)} - R_{\boldsymbol{\theta}}^T \mathbf{v} \|_2  }{ \| R_{\boldsymbol{\theta}}^T \mathbf{v} \|_2  } < 10^{-5}.
\end{align}
Since the preconditioner involves the DCT, an operation that is particularly efficient to perform on GPU computing architectures, we have also shown the wall-clock time results using both CPU and GPU computing architectures. The CPU computations were performed on a 2018 MacBook Pro with a 2.3 GHz Quad-Core Intel Core i5 processor and 8 GB of memory, and the GPU computations were performed on a computing cluster equipped with a single NVIDIA Tesla V100-SXM2-32GB GPU.

In light of \cref{fig:lapinv_results}, clearly the spectral preconditioning method drastically decreases the number of CG iterations needed and reduces the wall clock time required to compute the matrix-vector product with the pseudoinverse. This observation is more pronounced for extremely large $N$ and for the GPU computing architecture. In fact, the wall-clock time needed to compute the pseudoinverse with the spectral preconditioner and a GPU architecture is less than a tenth of a second for grids with dimensions as large as $1000 \times 1000$. For a comparison with a direct method as a benchmark, we have included results for the $\delta
$-approximation given in \cref{eq:approx_pseudoinverse}, which we implement using a banded Cholesky factorization costing $\mathcal{O}(N^{3/2})$. We observe that the Cholesky method is the preferred method for $N$ corresponding to grids with dimension smaller than about $56 \times 56$, but quickly becomes infeasible for larger grids due to the growth in cost.

\begin{figure}[tb]
    \centering
    \includegraphics[width=.99\textwidth]{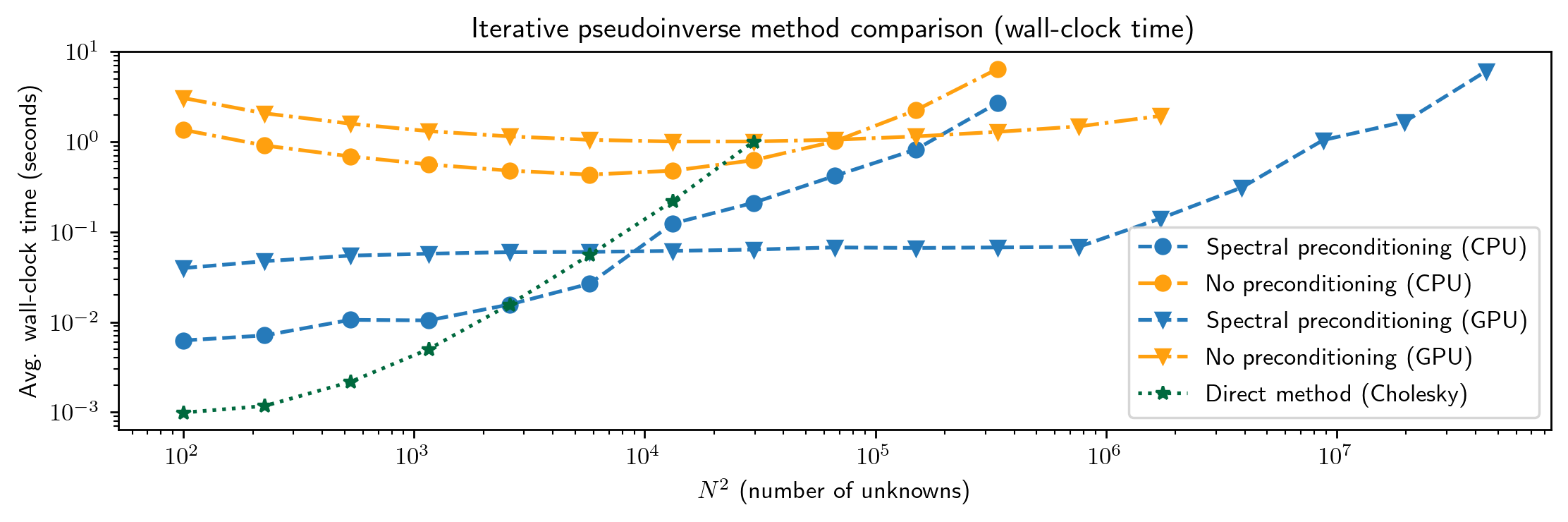}
    \caption{A comparison of the performance (in terms of wall-clock time) of iterative methods for computing matrix-vector products of the form $R_{\boldsymbol{\theta}}^\dagger \mathbf{v}$. We compare the CG method with and without the spectral preconditioner, carried out on both a CPU and GPU. We also provide results for a direct method using a banded Cholesky factorization as a benchmark.}
    \label{fig:lapinv_results}
\end{figure}

\section{Proof of \cref{thm:main_convexity}}
\label{app:main_convexity_convergence} We first re-write $\mathcal{G}$ in \cref{eq:IAS_G_with_noise_variance} as 
\begin{equation}\label{eq:IAS_noise_convexity1}
    \mathcal{G}( \mathbf{x}, \boldsymbol{\theta}, \nu ) 
        = \mathcal{G}_0( \mathbf{x}, \nu )
        + \mathcal{G}_1( \mathbf{x}, \boldsymbol{\theta} ), 
\end{equation}
with $\mathcal{G}_0$ and $\mathcal{G}_1$ respectively given by 
\begin{equation}\label{eq:IAS_noise_convexity2}
\begin{aligned} 
    \mathcal{G}_0( \mathbf{x}, \nu ) 
	& = \frac{1}{2 \nu} \| F \mathbf{x} - \mathbf{y} \|_2^2 
        + \left( \frac{\nu}{\tilde{\vartheta}} \right)^{\tilde{r}}  
        - \left( \tilde{r} \tilde{\beta} - [M+2]/2 \right) \log( \nu ) + \delta_{\mathbb{R}_{+}}(\nu), \\ 
    \mathcal{G}_1( \mathbf{x}, \boldsymbol{\theta} ) 
	& = \frac{1}{2} \| D_{\boldsymbol{\theta}}^{-1/2} R \mathbf{x} \|_2^2 
        + \sum_{i=1}^K \left( \frac{\theta_i}{\vartheta_i} \right)^{r} 
        - \left( r \beta - 3/2 \right) \sum_{i=1}^K \log( \theta_i ) + \delta_{\mathbb{R}_{+}^K}(\boldsymbol{\theta}). 
\end{aligned}
\end{equation} 
Observe that $\mathcal{G}$ is convex if $\mathcal{G}_0$ and $\mathcal{G}_1$ are convex. 
We proceed by addressing the convexity of $\mathcal{G}_1$ and $\mathcal{G}_0$ in \cref{lem:noise_convexity1,lem:noise_convexity2}, respectively, and combine these results to show \cref{thm:main_convexity}.

\begin{lemma}\label{lem:noise_convexity1}
    Let $r \in \R \setminus\{ 0 \}$ and $\beta > 0$. 
    Furthermore, let $\mathcal{G}_1(\mathbf{x},\boldsymbol{\theta})$ be the objective function in \cref{eq:IAS_noise_convexity2} and let $\eta = r \beta - 3/2$. 
    \begin{enumerate}
	   \item 
	   If $r \geq 1$ and $\eta > 0$, then $\mathcal{G}_1(\mathbf{x},\boldsymbol{\theta})$ is globally strictly convex. 

	   \item 
	   If $0 < r < 1$ and $\eta > 0$, or, if $r < 0$, then $\mathcal{G}_1(\mathbf{x},\boldsymbol{\theta})$ is locally convex at $(\mathbf{x}, \boldsymbol{\theta})$  provided that
	   \begin{equation}
		  \theta_i < \vartheta_i \left( \frac{\eta}{r | r - 1 |} \right)^{1/r}, \quad i=1,\ldots,K.
	   \end{equation}
	
    \end{enumerate} 	
\end{lemma}

\begin{proof} The proof follows the same logic as that of \cref{thm:previous_convexity_result}, originally given in \cite{calvetti2020sparse}.
    For completeness, we provide the most important steps here. 
    Since $\mathcal{G}_1$ is smooth on $\operatorname{Int}(\operatorname{dom}(\mathcal{G}_1))$, $\mathcal{G}_1$ is convex if and only if its Hessian is positive definite. 
    To this end, let $\mathbf{u} \in \R^N$ and $\mathbf{v} \in \R^K$ and $\small \mathbf{w} = [ \mathbf{u}; \mathbf{v} ]$. 
    To determine when the Hessian $H_{\mathcal{G}_1}(\mathbf{x}, \boldsymbol{\theta})$ is positive definite, we express the relevant quadratic form as
    \begin{eqnarray*} 
    \mathbf{w}^T H_{\mathcal{G}_1}(\mathbf{x}, \boldsymbol{\theta}) \mathbf{w}  
            = \sum_{i=1}^K \theta_i^{-1} \left( [R \mathbf{u}]_i - \theta_i^{-1} v_i [R \mathbf{x}]_i \right)^2 + \sum_{i=1}^K v_i^2 \left( \eta \theta_i^{-2} + \frac{r (r - 1)}{\vartheta_i^{r}} \theta_i^{r - 2} \right).
    \end{eqnarray*}
    Clearly, the first term on the right-hand side of the last equation are always non-negative, while the last term is non-negative provided that $\eta \theta_i^{-2} + \frac{r (r - 1)}{\vartheta_i^{r}} \theta_i^{r -2} > 0$ for $i = 1, \ldots, K$, yielding the assertion.
\end{proof}

\begin{lemma}\label{lem:noise_convexity2}
    Let $\tilde{r} \in \R \setminus\{ 0 \}$ and $\tilde{\beta} > 0$. 
    Furthermore, let $\mathcal{G}_0(\mathbf{x},\nu)$ be the objective function in \cref{eq:IAS_noise_convexity2} and let $\tilde{\eta} = \tilde{r} \tilde{\beta} - [M+2]/2$. 
    \begin{enumerate}
        \item 
	  If $\tilde{r} \geq 1$ and $\tilde{\eta} > 0$, then $\mathcal{G}_0(\mathbf{x},\nu)$ is globally strictly convex. 

	  \item 
	  If $0 < \tilde{r} < 1$ and $\tilde{\eta} > 0$, or, if $\tilde{r} < 0$, then $\mathcal{G}_0(\mathbf{x},\nu)$ is locally convex at $(\mathbf{x}, \nu)$ provided that
	  \begin{equation}
            \nu < \tilde{\vartheta} \left( \frac{ \tilde{\eta} }{ \tilde{r} | \tilde{r} - 1 |} \right)^{1/\tilde{r}}.
	  \end{equation}
    \end{enumerate} 	
\end{lemma}

\begin{proof} The proof follows the same logic as before.
    This time, the relevant quadratic form can be expressed as 
    \begin{equation}\label{eq:IAS_noise_convexity4} 
    \begin{aligned}
        \mathbf{w}^T H_{\mathcal{G}_0}(\mathbf{x}, \nu) \mathbf{w}  
            = \nu^{-1} \| F \mathbf{u} - v \nu^{-1} (F \mathbf{x} - \mathbf{y})  \|_2^2 + v^2 \left( \tilde{\eta} \nu^{-2} + \frac{ \tilde{r} ( \tilde{r} - 1)}{ \tilde{\vartheta}^{ \tilde{r} }} \nu^{ \tilde{r} - 2} \right).
    \end{aligned}
    \end{equation}
    The first term at the right-hand side of \cref{eq:IAS_noise_convexity4} is always non-negative, while the second term is non-negative if $\tilde{\eta} \nu^{-2} + \frac{ \tilde{r} ( \tilde{r} - 1)}{ \tilde{\vartheta}^{\tilde{r}}} \nu^{ \tilde{r} - 2} > 0$, yielding the assertion. 
\end{proof}

Finally, by combining \cref{lem:noise_convexity1,lem:noise_convexity2}, we get convexity conditions for $\mathcal{G}( \mathbf{x}, \boldsymbol{\theta}, \nu ) = \mathcal{G}_0( \mathbf{x}, \nu ) + \mathcal{G}_1( \mathbf{x}, \boldsymbol{\theta} )$,  therefore proving \cref{thm:main_convexity}.

\bibliographystyle{siamplain}
\bibliography{references}

\end{document}